\documentclass{article}
\usepackage{amssymb}	
\usepackage{amsmath}    
\usepackage{longtable}  
\usepackage{float}      
\usepackage{latexsym}	
\usepackage[a4paper]{geometry}

\usepackage{graphicx}
\usepackage{comment}
\usepackage{here}

\usepackage{lineno}

\newcommand*{\AN}[1]{(#1)}
\newcommand*{\BL}[1]{[#1]}

\usepackage{atbegshi}
\usepackage[
bookmarks=true,
bookmarksnumbered=true, 
bookmarkstype=toc,
setpagesize=false,
pdftitle={An Infeasible Interior-point Arc-search Algorithm for Nonlinear Constrained Optimization},
pdfauthor={Makoto Yamashita, Einosuke Iida, and Yaguang Yang },
pdfkeywords={Infeasible interior-point method, arc-search, nonlinear, nonconvex, constrained optimization.
}
]{hyperref}

\usepackage{color}
\usepackage[normalem]{ulem}

\newtheorem{algorithm}{Algorithm}[section]

\newtheorem{theorem}{Theorem}[section]
\newtheorem{lemma}{Lemma}[section]

\newtheorem{remark}{Remark}[section]

\newenvironment{proof}[1][Proof:]{\begin{trivlist} 
\item[\hskip \labelsep {\bfseries #1}]}{\end{trivlist}} 
\newcommand{\qed}{\nobreak \ifvmode \relax \else \ifdim\lastskip<1.5em \hskip-\lastskip \hskip1.5em plus0em minus0.5em \fi \nobreak \vrule height0.75em width0.5em depth0.25em\fi}

\def\R{{\mathbb{R}}}

\def\T{{\rm T}}
\def\D{{\cal D}}

\newif\iffigure
\figuretrue

\newif\iffigureeps
\figureepsfalse

\title{An Infeasible Interior-point Arc-search Algorithm for Nonlinear Constrained Optimization}
\author{
Makoto Yamashita\thanks{
	Department of Mathematical and Computing Science, 
	Tokyo Institute of Technology, Tokyo, Japan. Email: Makoto.Yamashita@c.titech.ac.jp. 
	His research was partially supported by JSPS KAKENHI (Grant Number: 18K11176).},	
	Einosuke Iida\thanks{
Department of Mathematical and Computing Science, 
Tokyo Institute of Technology, Tokyo, Japan.}, and 
Yaguang Yang\thanks{US NRC, Office of Research, 21 Church 
Street, Rockville, 20850. Email: yaguang.yang@verizon.net.},  
}
\date{October 28, 2020}

\begin{document}

\maketitle    

\begin{abstract}
In this paper, we propose an infeasible arc-search 
interior-point algorithm for solving nonlinear programming
problems. Most algorithms based on interior-point methods are categorized as line search, since they compute a next iterate 
on a straight line determined by a search direction 
which approximates the central path.
The proposed arc-search  interior-point 
algorithm uses an arc for the approximation.
We discuss convergence properties of the proposed algorithm.
We also conduct numerical experiments on the CUTEst 
benchmark problems and compare the performance of the 
proposed arc-search algorithm with that of a line-search 
algorithm. Numerical results indicate that the proposed 
arc-search algorithm reaches the optimal solution using less 
iterations but longer time than a line-search algorithm.
A modification that leads to a faster arc-search algorithm 
is also discussed.
\end{abstract}


\section{ Introduction}
Since great successes for linear programming (LP) problems \cite{wright97,ye97}, 
the interior-point methods 
have been extended to nonlinear programming
problems (NLPs)
\cite{bgn00,bhn99,ettz96,fg98,nww09,tp98,twbul02,uuv04,vs99}.
Almost all known strategies developed for LPs
were proposed for NLP formulated in different forms.
The most general form for NLP was considered in 
\cite{bgn00,bhn99,ettz96,fg98,tp98,twbul02,vs99}, while some
special form was discussed in \cite{nww09,uuv04}. 
Byrd et al.~\cite{bgn00,bhn99} handled the equality constraints ``as is'' in the papers, 
Vanderbei and Shanno~\cite{vs99} split the equality constraints into inequality constraints, and  
Forsgren and Gill~\cite{fg98} introduced a quadratic penalty function.
To analyze the convergence,
trust-region mechanisms were examined in
\cite{bgn00,bhn99}, and line-search strategies were also employed in 
\cite{ettz96,fg98,nww09,tp98,twbul02,uuv04,vs99}.

In the viewpoint of iterative methods, the interior-point methods can be classified 
into two groups by initial points;
``feasible'' interior-point methods \cite{fg98,twbul02}, which are easier to analyze but 
needs a ``phase-I'' process to find a feasible initial point, 
and ``infeasible'' interior-point 
methods \cite{bgn00,bhn99,ettz96,nww09,uuv04,vs99}, which 
do not need a feasible initial point but their convergence 
analysis is more difficult and their assumptions are more demanding.
From extensive numerical experience 
on interior-point methods for LPs in \cite{lms91,lms92a,Mehrotra92,yang17}, 
infeasible interior-point methods can be considered as a better 
strategy than feasible interior-point methods for NLPs. 

The central path plays an important part in the interior-point methods.
In particular, its accumulation point is an optimal solution,
thus the path-following type interior-point methods numerically trace
the central path and reach the optimal solution.
Most of the path-following type interior-point methods approximate 
the central path with a line determined by the search direction,
but the central path itself is usually not a straight line but a curve.

Recently, many researchers pay attention to arc-search 
interior-point methods. 
Yang~\cite{yang13} proposed the original arc-search 
interior-point method  for LPs.
The main idea in the arc-search methods is to approximate the central path 
with an arc of part of an ellipse 
and find the next iterate on the arc. Since the central path is usually
a curve, the arc can fit it more appropriately 
than the line. Yang and Yamashita~\cite{yy18} reported that 
an arc-search interior-point algorithm performed better than a 
line-search type interior-point algorithm for LPs.
The merit of the arc-search strategy is well demonstrated
in \cite{yang18} where an arc-search algorithm achieves
the best polynomial bound of $\mathcal{O}(\sqrt{n}\log{1/\epsilon})$
for all interior-point methods, feasible or infeasible, and is 
numerically competitive to the well-known Mehrotra’s algorithm.
The arc-search type methods are already extended to convex 
quadratic programming~\cite{yang13}, semidefinite 
programming~\cite{zyzlh19}, symmetric programming~\cite{ylz17}, 
and linear complementarity problems~\cite{kheirfam17}.

In this paper, we examine an extension of an infeasible 
arc-search interior-point algorithm to NLPs.
We discuss the convergence property of the proposed 
arc-search algorithm under mild conditions.
Compared to existing extensions above, the extension 
to NLPs is not simple due to their complicated structures. 
To show the convergence property, we introduce a merit function 
that measures a deviation from the KKT conditions.
We also discuss the analytical formula for the step angle. 

To verify the numerical performance of the proposed arc-search algorithm,
we conducted numerical experiments on the CUTEst problems~\cite{gould2015cutest}.
The results showed that the proposed algorithm required 
fewer iterations than a line-search algorithm.
In particular, the reduction in the number of iterations was clearer for quadratic-constrained quadratic programming (QCQP) problems.
We also examined a computation time reduction by a modification on the second derivative.

The remainder of the paper is organized as follows. Section~\ref{sec:description}
introduces the problem. In Section~\ref{sec:algorithm}, we describe the proposed arc-search
algorithm, and in Section~\ref{sec:convergence}, we discuss its convergence properties. Section~\ref{sec:experiments} provides the numerical results and discusses the modification on the second derivatives.
Finally, Section~\ref{sec:conclusion} gives the conclusions of this paper.

\section{Problem description}\label{sec:description}

We consider a general nonlinear programming problem:
\begin{align}
\begin{array}{rcl}
\min &:& f(x) \\
\textrm{s.t.} &:& h(x) = 0, \ g(x) \ge 0,
\end{array}
\label{NP}
 \end{align}
where $f: \R^n \rightarrow \R$,  $h: \R^n \rightarrow \R^m$, $m<n$, and 
$g: \R^n \rightarrow \R^p$. To simplify the latter discussions, we assume $p \ge 1$.
The decision variable is $x \in \R^n$.

For the inequality constraints $g(x) \ge 0$, we convert them
into equality constraints introducing a slack vector $s \in \R^p$ as follows:
\begin{align}
\begin{array}{rcl}
	\min &:&  f(x)  \\
	\textrm{s.t.} &:&  h(x)=0, \
	 g(x) - s = 0, \ s \ge 0.
	 \end{array}
	\label{NP1}
	\end{align}
    
Throughout the paper, a tuple is used to denote 
a concatenation of vectors, for example,
$(x,y,z)$ stands for $(x^{\T},y^{\T},z^{\T})^{\T}$,
where the superscript $\T$ is the transpose of a vector or a matrix.
For a vector $x \in \R^n$, 
$\D(x) \in \R^{n \times n}$ is a diagonal matrix whose 
diagonal elements form $x$,
and  $\min(x)$ is the minimum value in $x$.
Let $\R_+^n$ ($\R_{++}^n$)  denote the space of nonnegative vectors 
(positive vectors, respectively),
and $e$  denote a vector of all ones with 
appropriate dimension. 
    
For (\ref{NP1}), we introduce Lagrangian multipliers 
$y \in \R^m, w \in \R^p$ and $z \in \R^p$ and use 
$v  = (x,y,w,s,z) \in \R^{n+m+3p} $ to denote the tuple of decision variables 
and multipliers. Then, the Lagrangian function for (\ref{NP1}) is 
\begin{equation*}
L(v)=f(x)+y^{\T}h(x)-w^{\T}(g(x)-s)-z^{\T}s,
\label{lagrangian1}
\end{equation*}
and its gradients with respect to $x$ and $s$ are
\begin{equation}
\nabla_x L(v)=\nabla f(x)+\nabla h(x) y-\nabla g(x) w,
\hspace{0.1in} \nabla_s L(v)=w-z,
\label{dlagrangian}
\end{equation}
respectively.
The notation related to derivatives in this paper are summarized in \ref{section:derivatives}.
The KKT conditions for (\ref{NP1}) are
\begin{align}
 F(v)  =  0,  \
 (w, s, z) \in \R_+^{3p}, 
\label{KKT1}
\end{align}
where $F:\R^{n+m+3p} \to \R^{n+m+3p}$ is defined by 
\begin{equation*}
F(v) = \left[ \begin{array}{c}
\nabla_x L(v)  \\
h(x)  \\
g(x)-s \\
w-z  \\
\D(z)s
\end{array} \right].
\label{defineF}
\end{equation*}
The Jacobian of $F$ is given by
\begin{align*}
F' (v) = \left[ \begin{array}{ccccc}
\nabla_x^2 L(v) & \nabla h(x) & -\nabla g(x) & 0 & 0 \\
\left( \nabla h(x)\right)^{\T} & 0 & 0 & 0 & 0  \\
\left( \nabla g(x)\right)^{\T} & 0 & 0 & -I & 0 \\
0  & 0  &  I  &  0  &  -I  \\
0 & 0 & 0  &  \D(z) & \D(s)
\end{array} \right]. 
\label{firstJacobian}
\end{align*}
The index set of active inequality constraints at $x \in \R^n$ is denoted by
\begin{equation*}
I(x)=\left\{i \in \{1,\ldots,p\} : g_i(x)=0 \right\}.
\label{index}
\end{equation*}

Similarly to \cite{ettz96}, we make the following standard 
assumptions for (\ref{NP}).

\vspace{0.05in}
\noindent
{\bf Assumptions}
\begin{itemize}
\item[(A1)] Existence. There exists $v^*=(x^*,y^*,w^*,s^*,z^*)$, 
an optimal solution of (\ref{NP1}) and its associate 
multipliers. The KKT conditions (\ref{KKT1}) hold at
any optimal solution.
\item[(A2)] Smoothness. $f(x)$ is differentiable up to the third order, and 
$h(x)$ and $g(x)$ are up to the second order.
In addition, $f(x)$, $g(x)$, and $h(x)$ 
are locally Lipschitz continuous at $x^*$.
\item[(A3)]  Regularity. The set $\{ \nabla h_j(x^*) : j =1, \ldots, m\} \cup
\{ \nabla g_i(x^*)  : i \in I(x^*) \}$
is linearly independent.
\item[(A4)] Sufficiency. For all $\eta  \in \R^n \backslash \{0\}$,
we have $ \eta^{\T} \nabla_x^2 L(v^*)  \eta >0$.
\item[(A5)]  Strict complementarity. For each
$i \in \{ 1, \ldots, p \}$, we have $z_i^*+s_i^* >0$
and $z_i^* s_i^* =0$.
\end{itemize}

From these assumptions, we can guarantee the nonsingularity 
of the Jacobian matrix at the optimal solution $v^*$.

\begin{theorem}\label{nonsingular}
If  (A1), (A3), (A4), and (A5) hold, the Jacobian 
matrix $F'(v^*)$ is nonsingular.
\end{theorem}
\begin{proof} 
This is a well-known result and its proof is omitted.
\hfill \qed
\end{proof}

\section{The arc-search algorithm}\label{sec:algorithm}
Given a point  
${v}=({x},{y},{w},{s},{z})$ and $t>0$, 
let ${v}\BL{t} =({x}\BL{t},{y}\BL{t},{w}\BL{t},{s}\BL{t},{z}\BL{t}) \in 
\R^n \times \R^m \times \R^p \times \R^p\times \R^p$ be the solution of 
the perturbed KKT conditions $F({v}\BL{t}) = t F({v})$ with 
nonnegative conditions, that is, ${v}\BL{t}$ satisfies
\begin{eqnarray}
\left[ \begin{array}{l}
\nabla_x L({v}\BL{t}) \\
h({x}\BL{t}) \\
g({x}\BL{t})-s\BL{t} \\
\nabla_s L({v}\BL{t})  \\
\D({z}\BL{t}){s}\BL{t}
\end{array} \right]= \left[ \begin{array}{l}
t \nabla_x L({v}) \\
t h({x}) \\
t (g({x})-{s}) \\
t \nabla_s L({v})  \\
t \D({z}) {s} 
\end{array} \right],
\hspace{0.1in}
({w}\BL{t}, {s}\BL{t}, {z}\BL{t}) \in \R_+^{3p}.
\label{KKTcurve}
\end{eqnarray}
Note that under some mild conditions that will be introduced as
(B1)-(B4) later, ${v}\BL{t}$ is uniquely determined for each 
$t \in (0,1]$ due to the implicit 
function theorem and Lemma~\ref{FpInv} below, thus we define 
$C = \left\{{v}\BL{t} \in \R^{n+m+3p} : t \in (0,1] \right\}$.
Since the right-hand-side of (\ref{KKTcurve}) converges to zeros 
when $t \to 0$, 
$v\BL{t}$ also converges to 
a point that satisfies the KKT conditions \eqref{KKT1} under the mild condition.

The main  strategy of the arc-search algorithm is to 
approximate $C$ with an ellipse.
We denote the ellipse by
\begin{equation}
{\cal E}=\lbrace v\AN{\alpha}: 
v\AN{\alpha} =
\vec{a}\cos(\alpha)+\vec{b}\sin(\alpha)+\vec{c}, \alpha \in [0, 2\pi] \rbrace,
\label{ellipse}
\end{equation}
where $\vec{a} \in \R^{n+m+3p}$ and $\vec{b} \in \R^{n+m+3p}$ are 
the axes of the ellipse, and $\vec{c} \in \R^{n+m+3p}$ is its center. 
The ellipsoid approximation of $C$ will be given 
in Theorem~\ref{theorem:ellip} below.
Before formally stating Theorem~\ref{theorem:ellip}, we 
introduce notation on the derivatives.
The first-order derivative at $t=1$ along $C$ is given by 
$F' (v\BL{t})\vert_{t=1} = F'({v})$.
Let  ${\mu}=\frac{{z}^{\T}{s}}{p}$ be the duality 
measure at ${v}$ and $\sigma \in (0,1)$ be a parameter.
We use \begin{align*}
\dot{v} = (\dot{x}, \dot{y}, \dot{w}, \dot{s},  \dot{z})
\end{align*}
to denote the solution of a modified Newton system
\begin{align*}
F'({v}) \dot{v} = F({v}) - \sigma {\mu} \bar{e}, 
\end{align*}
where $\bar{e} = (0,0,0,0,e)$ is the vector with $p$ ones at 
the bottom of the vector.
Here, we add $-\sigma {\mu} e$ to 
the last element in a similar way to the strategy used in \cite{Mehrotra92,yang17}.
This modification is applied to guarantee that a substantial segment of 
the ellipse satisfies $(s,z)>0$, thereby  
the step size along the ellipse is greater than zero. 
The system $F'({v}) \dot{v} = F({v}) - \sigma {\mu} \bar{e}$ is also written as
\begin{equation}
\left[ \begin{array}{ccccc}
\nabla_x^2 L({v}) & \nabla h({x}) & -\nabla g({x}) & 0  & 0 \\
\left( \nabla h({x})\right)^{\T} & 0 & 0 & 0 & 0  \\
\left( \nabla g({x})\right)^{\T} & 0 & 0 & -I  & 0 \\
 0  & 0  &  I  &  0  &  -I  \\
0 & 0  & 0 & \D({z}) & \D({s})
\end{array} \right]
\left[ \begin{array}{c}
\dot{x} \\ \dot{y} \\ \dot{w} \\ \dot{s}  \\ \dot{z}
\end{array} \right]
= \left[ \begin{array}{l}
\nabla_x L({v}) \\
h({x}) \\
g({x})-{s} \\
{w}-{z}  \\
\D({z}) {s} -\sigma {\mu} e 
\end{array} \right].
\label{firstOrder}
\end{equation}

Next, for the second-order derivative at 
$t=1$ along the curve, we define $\ddot{v} = (\ddot{x}, \ddot{y}, \ddot{w}, \ddot{s}, \ddot{z})$ 
as the solution of the following system:
{\footnotesize
\begin{eqnarray}
 \left[ \begin{array}{ccccc}
\nabla_x^2 L(v) & \nabla h(x) & -\nabla g(x) & 0  & 0 \\
\left( \nabla h(x)\right)^{\T} & 0 & 0 & 0 & 0  \\
\left( \nabla g(x)\right)^{\T} & 0 & 0 & -I  & 0 \\
0  & 0  &  I  &  0  &  -I  \\
0 & 0 & 0  & \D({z}) & \D({s})
\end{array} \right]
\left[ \begin{array}{c}
\ddot{x} \\ \ddot{y} \\ \ddot{w} \\ \ddot{s} \\ \ddot{z}
\end{array} \right]
 =  \left[ \begin{array}{l}
-(\nabla_x^3 L(v))\dot{x} \dot{x}
-2(\nabla_x^2 h(x))\dot{y} \dot{x}

+2(\nabla_x^2 g(x))\dot{z} \dot{x} \\
-(\nabla_x^2 h(x))^{\T} \dot{x} \dot{x} \\
-(\nabla_x^2 g(x))^{\T} \dot{x} \dot{x} \\
0  \\
-2\D(\dot{z}) \dot{s} 
\end{array} \right].
\label{secondOrder}
\end{eqnarray}
}
The formula for computing the elements in the right-hand-side can be found in \ref{section:derivatives}.

We call $\dot{v} = (\dot{x}, \dot{y}, \dot{w} ,\dot{s}, \dot{z})$ in (\ref{firstOrder}) 
and $\ddot{v} = (\ddot{x}, \ddot{y}, \ddot{w},\ddot{s},\dot{z})$  in 
(\ref{secondOrder}) the first derivative 
and the second derivative of the ellipse ${\cal E}$, respectively.
Using $\dot{v}$ and $\ddot{v}$, we can approximate 
$C$ at $t=1$ by an ellipse (\ref{ellipse}) 
that has the explicit form as in the following theorem.
Note that, ${v}\BL{t}$ passes ${v}$ at $t=1$ while
${v}\AN{\alpha}$ passes $v$ at $\alpha = 0$, that is,
${v}\BL{1} = {v}\AN{0} = {v}$.

\begin{theorem}\label{theorem:ellip}
	\upshape{\cite{yang13}}
Suppose that an ellipse ${\cal E}$ 
of form \textrm{(\ref{ellipse})} passes through 
a point ${v}$ at $\alpha=0$, and its first and second 
order derivatives at $\alpha=0$ are
$\dot{v}$ 
and $\ddot{v}$, respectively. Then 
${v}\AN{\alpha} = ({x}\AN{\alpha}, {y}\AN{\alpha}, {w}\AN{\alpha}, {s}\AN{\alpha}, {z}\AN{\alpha})$ 
of ${\cal E}$ is given by
\begin{align}
{v}\AN{\alpha} = {v} - \dot{v}\sin(\alpha)+\ddot{v}(1-\cos(\alpha))
\label{vAlpha}. 
\end{align}
\end{theorem}

The computation of (\ref{vAlpha}) 
can be simplified as  the following lemma.
\begin{lemma}\label{eqwz}
If ${v}$ satisfies ${w}={z}$, then ${w}\AN{\alpha}={z}\AN{\alpha}$ 
holds for any $\alpha \in \R$.
\end{lemma}
\begin{proof}
From the fourth row of (\ref{firstOrder}), we have 
$\dot{w}-\dot{z}={w}-{z} = 0$. Similarly, the fourth
row in (\ref{secondOrder}) leads to 
$\ddot{w}-\ddot{z}=0$. Therefore, the formula
(\ref{vAlpha}) gives the lemma.
\hfill \qed
\end{proof}

To reach an optimal solution that satisfies the KKT conditions  (\ref{KKT1})
along the ellipse ${\cal E}$,
the merit function defined by
\begin{equation}
\phi(v) = \| F(v) \|^2
\end{equation}
should sufficiently decrease at ${v}\AN{\alpha}$ for 
some constants $\beta \in (0, \frac{1}{2}]$,
$\sigma \in \left(\bar{\sigma}, 1 \right)$,  and step angle 
$\alpha \in (0, \pi/2]$, \textit{i.e.},
\begin{equation*}
\phi({v}\AN{\alpha}) = \| F({v}\AN{\alpha}) \|^2 
\le \phi({v}) (1 - 2\beta  (1-\sigma) \sin(\alpha))
	< \phi({v})
\label{merit}
\end{equation*}
which will be proved later. 
Using the ellipsoid approximation and the merit function 
$\phi(v)$, we give a framework of  the proposed arc-search 
algorithm.

\begin{algorithm}\label{algorithm:arc-search} {\bf (an infeasible arc-search interior-point algorithm
)} \newline\newline
    \indent Parameters: $\epsilon>0$, $\delta>0$,  
    $\beta \in (0, \frac{1}{2}]$, $\bar{\sigma} \in (0, \frac{1}{2})$, and $\gamma_{-1} \in [0.5,1)$.
    \newline\indent Initial point: $v^0 = (x^0,y^0,w^0,s^0,z^0)$ 
    such that $(w^0,s^0,z^0) \in \R_{++}^{3p}$ and $w^0=z^0$.
    \newline
    \newline
    \indent {\bf for} iteration $k=0,1,2,\ldots$
    \begin{itemize}
        \item[] Step 1:  If $\phi(v^k) \le \epsilon$, stop.
        \item[] Step 2: Calculate $\nabla_x L(v^k)$, $h(x^k)$, $g(x^k)$, 
        $\nabla_x^2 L(v^k)$, $\nabla_x h(x^k)$, and $\nabla_x g(x^k)$.
        \item[] Step 3: Select $\sigma_k$ such that 
        $\bar{\sigma} \le \sigma_k < \frac{1}{2}$
        and let  
        $\dot{v}^k = (\dot{x}^k,\dot{y}^k,\dot{w}^k,\dot{s}^k,\dot{z}^k)$
        be the solution of (\ref{firstOrder}) at ${v} = v^k$.
        \item[] Step 4: Calculate 
        $\left(\nabla_x^3 L \right)\dot{x}\dot{x}$,
        $\left(\nabla_x^2 h \right)\dot{x}\dot{y}$, 
        $\left( \nabla_x^2 g \right) \dot{x}\dot{z}$, 
        $\left( \nabla_x^2 h \right)^{\T}  \dot{x}\dot{x}$, 
        $\left( \nabla_x^2 g \right)^{\T}  \dot{x}\dot{x}$, and 
        $\D(\dot{z})\dot{s}$.
        \item[] Step 5: Let 
        $\ddot{v}^k = (\ddot{x}^k,\ddot{y}^k,\ddot{w}^k,\ddot{s}^k,\ddot{z}^k)$
        be the solution of (\ref{secondOrder})  at ${v} = v^k$.
        \item[] Step 6:  Choose $\gamma_k$ such that $\frac{1}{2} \le \gamma_{k} 
        \le \gamma_{k-1}$. Find  appropriate 
        $\alpha_k > 0$ by (\ref{alphaK}) below using $\gamma_k$.
        \item[] Step 7:  Update $v^{k+1} = v^k\AN{\alpha_k} = v^k 
        - \dot{v}^k \sin(\alpha_k) + \ddot{v}^k (1-\cos(\alpha_k))$ .
    \end{itemize}
    \indent\indent {\bf end (for)} 
    \hfill \qed
    \label{mainAlgo1}
\end{algorithm}

As an interior-point method, we should choose the step angle $\alpha_k \in (0, \pi/2] $ which satisfies
the following conditions: 
\begin{itemize}
\item[(C1)] $(w^k\AN{\alpha_k}),s^k\AN{\alpha_k},z^k\AN{\alpha_k}) \in \R_{++}^{3p}$.
\item[(C2)] The generated sequence $\{v^k\}$
should be bounded. 
\item[(C3)] $\phi(v^{k+1}) = \phi(v^k\AN{\alpha_k}) < \phi(v^k)$.
\end{itemize}

We can realize (C1) by 
a process developed in \cite{yang17}. Due to Lemma~\ref{eqwz}, we can always have $z^k\AN{\alpha}=w^k\AN{\alpha}$. 
Fix  a small $\delta \in (0,1)$.
We will select the largest $\tilde{\alpha}$ such that all 
$\alpha \in [0, \tilde{\alpha}]$ satisfy 
\begin{subequations}\label{positive}
\begin{align}
w^k\AN{\alpha}&=w^k - \dot{w}^k\sin(\alpha)+\ddot{w}^k(1-\cos(\alpha)) 
 \ge \delta w^k, \label{key1a} \\
s^k\AN{\alpha}&=s^k - \dot{s}^k\sin(\alpha)+\ddot{s}^k(1-\cos(\alpha)) 
 \ge \delta s^k. \label{key1b} 
\end{align}
\label{analyticArc}
\end{subequations}
To this end, for each 
$i \in \lbrace 1,\ldots, p \rbrace$, we select the largest
$\alpha_{w_i}^k$ such that 
the $i$th inequality of (\ref{key1a}) holds
for any $\alpha \in [0, \alpha_{w_i}^k ]$
and the largest 
$\alpha_{s_i}^k$ such that 
the $i$th inequality of (\ref{key1b}) holds for any $\alpha \in [0, \alpha_{s_i}^k ]$ .
We then define 
\begin{equation}
\tilde{\alpha}_k=\min_{i \in \lbrace 1,\ldots, p \rbrace}
\lbrace \min \{\alpha_{w_i}^k, \alpha_{s_i}^k, \frac{\pi}{2} \} \rbrace.
\label{alpha}
\end{equation}
The largest $\alpha_{w_i}$ and $\alpha_{s_i}$ can be given in analytical 
forms. See \ref{section:ComputeAlpha}.

For (C2), we define 
\begin{equation}
\hat{m}_k(\alpha)  = \min (\D(z^k\AN{\alpha}) s^k\AN{\alpha} )- \gamma_k 
\min (\D(z^0) s^0) 
\frac{\phi (v\AN{\alpha})}{\phi (v^0)}.
\label{measurePos}
\end{equation}
If $\alpha_k$ is chosen such that $\hat{m}_k(\alpha_k) \ge 0$,
 $(w_k,s_k,z_k)$ should not 
approach to the boundary too fast,
and this guarantees (C2).
This essentially has the same effect as the wide neighborhood
of interior-point methods \cite{wright97}.
Here, we define
\[
\hat{\alpha}_k = \max\left\{ \alpha \in \left(0, \frac{\pi}{2}\right] : \hat{m}_k(\alpha) \ge 0 \right\}.
\]

Finally, to realize (C3), we present the following lemma. 
\begin{lemma}
	Let $\alpha \in (0, \frac{\pi}{2}]$, $\beta \in (0, \frac{1}{2}]$ and $\sigma \in \left(\bar{\sigma}, 1 \right)$. 
    Let $\mu = \frac{z^{\T}s}{p}$.
    If
	\begin{equation}
	\phi({v}\AN{\alpha}) \le \phi({v}) - \beta \sin(\alpha) \nabla_{\alpha} \phi({v}\AN{\alpha})|_{\alpha=0},
	\label{cond1}
	\end{equation}
	then
	\begin{equation}
	\phi({v}\AN{\alpha}) \le \phi({v}) (1 - 2\beta  (1-\sigma) \sin(\alpha))
	< \phi({v}).
	\label{conC1}
	\end{equation}
	\label{decrease}
\end{lemma}
\begin{proof}
	The right inequality in (\ref{conC1}) is clear for given 
	$\alpha, \beta$, and $\sigma$. The left inequality in (\ref{conC1}) 
	follows from a similar argument in \cite{ettz96}. 
	Since $\dot{v}$ is defined as the solution of  
        $F'({v}) \dot{v}=F({v})-\sigma {\mu} \bar{e}$ at $\mu = \frac{z^\T s}{p}$, we have
	\begin{equation}
	\nabla_{\alpha} \phi({v}\AN{\alpha})|_{\alpha=0} =2F({v})^{\T}F'({v}) 
        \dot{v}
	=2 F({v})^{\T}( F({v}) -\sigma {\mu} \bar{e})
	=2( \phi({v})-\sigma {\mu}^2/p),
	\label{tmp1}
	\end{equation}
	where the last equality is derived from  
        $F({v})^{\T} (\sigma {\mu} \bar{e}) = \sigma {\mu} 
	\sum_{i=1}^{p} {z}_i {s}_i = \sigma {\mu} {z}^{\T}{s}
	=\sigma {\mu}^2 /p$. 
	Since $|{z}^{\T}{s}| \le \sqrt{p} \| \D({z}) {s} \|_2$ and $p \ge 1$, 
	we have
	\[
	{\mu}^2/p =({z}^{\T}{s})^2 /p^2 \cdot (1/p) 
	\le \|\D({z}){s} \|_2^2 \cdot 1 \le \| F({v}) \|_2^2 = \phi({v}).
	\]
	Substituting this inequality into (\ref{tmp1}), we have 
\begin{equation}
\nabla_{\alpha} \phi({v}\AN{\alpha})|_{\alpha=0} \ge 2\phi({v}) (1-\sigma).
\label{objReduction}
\end{equation}
	From (\ref{cond1}), it holds 
	\[
	\phi({v}\AN{\alpha}) \le \phi({v}) - 2 \beta \sin(\alpha) \phi({v}) (1-\sigma)
	=\phi({v}) (1 - 2\beta  (1-\sigma) \sin(\alpha) ).
	\]
	This completes the proof.
	\hfill \qed
\end{proof}
We define 
$\check{\alpha}_k$ as the largest $\alpha$ that satisfies 
(\ref{cond1}), therefore, for a small constant parameter $\delta$, we define
\begin{equation}
\check{\alpha}_k = \max\left\{ \alpha \in \left(0, \frac{\pi}{2}\right] : 
\beta \sin(\alpha) \nabla_{\alpha} \phi({v^k}\AN{\alpha})|_{\alpha=0} >\delta
\right\}.
\label{alphaCheck}
\end{equation}

From these observation, the step angle in the $k$th iteration
should be taken as:
 \begin{equation}
 \alpha_k = \min\{\tilde{\alpha}_k, \hat{\alpha}_k, \check{\alpha}_k \} > 0.
 \label{alphaK}
 \end{equation}
We will show through the convergence analysis in the next section that
the sequence $\{\alpha_k\}$ is bounded below and away from zero
	during the iterations of algorithm.
A sequence $\{c_k\} \subset \R $ is said to be \textit{bounded below and away from zero
if there exists $\bar{c} > 0 $ such that $c_k \ge \bar{c}$
for all $k \ge 1$.}

\begin{remark}
(C1) is enforced by (\ref{analyticArc}), (C3) is proved 
to hold in Lemma \ref{decrease}, and (C2) will be 
proved to hold in the next setion.
\end{remark}

\section{Convergence analysis}\label{sec:convergence}

To discuss the global convergence of Algorithm \ref{mainAlgo1}, 
we define a set $\Omega(\epsilon)$ for $\epsilon > 0$  as follows:
\begin{equation*}
\Omega (\epsilon) = \left\{ v \in \R^{n+m+3p}: \epsilon \le \phi(v) \le \phi(v^0), 
\ \min (\D(s)z) \ge \frac{1}{2} \min (\D(s^0)z^0) 
\frac{\phi (v\AN{\alpha})}{\phi (v^0)}
\right\}.
\label{epsilonSet}
\end{equation*}

Some additional assumptions similar to the ones used in 
\cite{ettz96} are introduced.

\vspace{0.05in}
\noindent
{\bf Assumptions}
	\begin{itemize}
		\item[(B1)] In the set $\Omega(\epsilon)$, 
the columns of $\nabla h(x)$ are linearly 
		independent.
		\item[(B2)] The sequence $\{ x^k \} \subset \R^n$ is bounded.
		\item[(B3)] The matrix $\nabla_x^2 L(v)+\nabla g(x) \D(s)^{-1} \D(z) (\nabla g(x))^{\T}$ 
		is invertible for any $v$ in any compact subset of $\Omega(\epsilon)$.
		\item[(B4)]  Let $I_s^k$ be the index set  
		$\{ i: 1 \le i \le p, \,\,  s^k_i =0 \}$. Then, 
                the determinant of $(J^k)^{\T} J^k$ 
                is bounded below and away from zero, where  $J^k$ is a matrix 
                whose column vectors are composed of
		\[
		\{ \nabla h_j (x^k) : j = 1, \ldots, m\} \cup
		\{\nabla g_i(x^k): i \in I_s^k  \}.
		\]
	\end{itemize}
\begin{remark}
Note that if $v^k$ is close to $v^*$ for sufficiently large $k$, (B1) and (B4)  
automatically hold from (A3). (B3) also holds from (A4) for a small compact 
subset of $\Omega(\epsilon)$ around $v^*$. It is worthwhile to point out 
that Assumptions (B1), (B2), and (B3) are not more restrictive than the 
assumptions of (C1), (C2), and (C3) in \cite{ettz96}, which is a widely cited article.
\end{remark}

The convergence analysis is divided into a series of lemmas.
Through Lemma \ref{boundAboveBelow} to Lemma \ref{ettzTheorem},
we show that all the vectors and the matrices are bounded. Then, the positivenesses of 
$\tilde{\alpha}^k, \check{\alpha}^k$ and $\hat{\alpha}^k$ 
are guaranteed in Lemmas~\ref{barAlpha}, \ref{checkAlpha} and \ref{hatAlpha},
respectively. Using these lemmas, the convergence of Algorithm~\ref{mainAlgo1}
will be established in Theorem~\ref{global}.

\begin{lemma}\label{boundAboveBelow}
Assume that (B1)-(B4) hold. If the sequence $\{v^k\}$ satisfies  $\{v^k\} \subset \Omega(\epsilon)$
for some $\epsilon > 0$, 
then  $\{ v^k \}$ is bounded 
and  $\{(w^k, s^k, z^k)\} \subset \R_{++}^{3p}$
is bounded below and away from zero.
\end{lemma}
\begin{proof}
From (B2) and the continuity of $g$, the boundedness of $\{ x^k \}$ 
implies that $\{ g(x^k) \}$ is bounded. 
In view of Lemma \ref{decrease}, Step~6 of Algorithm~\ref{mainAlgo1}  guarantees that 
\eqref{conC1} holds, which indicates that
$\{ \phi(v^k)  \}$ is monotonically decreasing. Therefore,
$\| g(x^k) -s^k \|^2 \le \|F(v^k)\|^2 = \phi(v^k) \le \phi(v^0)$ is bounded. 
Since $\| s^k \| \le \| g(x^k) -s^k \| + \| g(x^k) \|$, we know that 
$\{s^k\}$ is bounded.

We prove that $\{z^k\}$ is also bounded.
In view of Lemma~\ref{eqwz}, we have $w^k = z^k$.
Suppose by contradiction that  $z_i^k  =
 w_i^k  \rightarrow \infty$ when $k\to \infty$ for some $i$.
Since $\{ \phi(v^k)\}$ is 
bounded as discussed above, $\{\nabla_x L(v^k)\}$ and $\{\D(z^k) s^k\}$ are bounded.  
Furthermore,  (B2) implies that $\{ \nabla f(x^k) \}$ is bounded.
Therefore, in view of (\ref{dlagrangian}),
\[
\| \nabla h(x^k) y^k-\nabla g(x^k) w^k \|
\le \| \nabla_x L(v^k) \|
+ \| \nabla f(x^k) \|
\]
is also bounded.
This indicates that 
$\| \nabla h(x^k) y^k-\nabla g(x^k) z^k \|$ is bounded
because of $w^k = z^k$. As $ w_i^k \rightarrow \infty$
implies $\| (y^k, w^k) \| \rightarrow \infty$, it holds 
\begin{equation}
\| \nabla h(x^k) y^k-\nabla g(x^k) w^k \|/ (\| (y^k,w^k) \|) 
\rightarrow 0.
\label{contradict}
\end{equation}
Let $(\hat{y}, \hat{w})$ be an accumulation point of
$\{ (y^k, w^k) / \| (y^k, w^k) \|\}$. Clearly $\| (\hat{y},\hat{w}) \| =1$. 
The boundedness of $\{\D(z^k)s^k\}$ 
implies that $\{z_i^k s_i^k\}$ is bounded for each $i$.
Since $\{s^k\}$ is bounded, we can take an accumulation point $\hat{s}$,
	and we define a set $I_s = \{i : 1 \le i \le p, \hat{s}_i = 0\}$.
Due to  $w_i^k = z_i^k$, 
$w_i^k \rightarrow \infty$ indicates
$\hat{s}_i = 0$, therefore, $i \in I_s$.
If $j \notin I_s$, then $w_j^k < \infty$, hence 
$\hat{w}_j = 0$. 
From (\ref{contradict}),  it holds that 
\[
 \nabla h(x^k) \hat{y}  -\nabla g(x^k) \hat{w}  =
\nabla h(x^k) \hat{y} - \sum_{i \in I_s} \nabla g_i(x^k) \hat{w}_i \to 0.
\]
Since $\| (\hat{y},\hat{w}) \| =1$, this contradicts with (B4). 
Therefore, $\{w^k\}$ and $\{z^k\}$ are bounded.

Since $\{v^k\} \subset \Omega(\epsilon)$, the sequence $\{ z_i^k s_i^k \}$
are all bounded below and away from zero for each $i=1, \ldots, p$;
more precisely, $z_i^k s_i^k \ge \frac{1}{2} \min(\D(z^0) s^0)
\frac{\phi(v^k)}{\phi(v^0)} \ge \frac{1}{2} \min(\D(z^0) s^0)
\frac{\epsilon}{\phi(v^0)}$ for each $i$.
Therefore,  $\{ z_i^k \}$ is bounded below and away from zero, since
$\{ s_i^k \}$ is bounded. 
Similarly, $\{ s_i^k \}$ is also bounded below and away from zero.

Finally, using (\ref{dlagrangian}) and (B1), we have
\[
y^k = ((\nabla h(x^k))^{\T}\nabla h(x^k))^{-1}(\nabla h(x^k))^{\T}
\left[ \nabla_x L(v^k) - \nabla f(x^k)
+ \nabla g(x^k) w^k
\right],
\]
hence, $\{y^k\}$ is bounded because $\{x^k\}$ and $\{w^k\}$ are bounded.
\hfill \qed
\end{proof}

The invertiblility of a block matrix guaranteed in  the following lemma 
will be used to show the boundedness 
of the inverse of the Jacobian $\{F'(v^k)\}$ in Lemma~\ref{lem:inverseFprime} below.
\begin{lemma}\label{blockInverse}\upshape{\cite{ls02}}
    Let $R$ be a block matrix
    \[
    R=\left[ \begin{array}{cc}
    A & B \\ C & D  \end{array}   \right].
    \]
    If $A$ and $D-CA^{-1}B$ are invertible,
    or $D$ and $A-BD^{-1}C$ are invertible, then $R$ is invertible.
\end{lemma}
\begin{lemma}\label{lem:inverseFprime}
    Assume that (B1)-(B4) hold. If $\{v^k\} \subset \Omega(\epsilon)$ for some $\epsilon > 0$, 
    then $\{ [F'(v^k)]^{-1} \}$ is bounded.
    \label{invertibility}
    \label{FpInv}
\end{lemma}
\begin{proof}

    We decompose $F'(v^k)$ into sub-matrices:
    \begin{align}
    F' (v^k) = \left[ \begin{array}{ccccc}
    \nabla_x^2 L(v^k) & \nabla h(x^k) & -\nabla g(x^k) & 0 & 0 \\
    \left( \nabla h(x^k)\right)^{\T} & 0 & 0 & 0 & 0  \\
    \left( \nabla g(x^k)\right)^{\T} & 0 & 0 & -I & 0 \\
    0  & 0  &  I  &  0  &  -I  \\
    0 & 0 & 0  &  \D(z^k) & \D(s^k)
    \end{array} \right]
    = \left[ \begin{array}{cc}A^k & B^k \\ C^k & D^k \end{array}\right]
    \end{align}
    where 
    \begin{align*}
    & A^k = \left[ \begin{array}{cc}
    \nabla_x^2 L(v^k)   &  \nabla h(x^k) \\
    (\nabla h(x^k))^\T  & 0
    \end{array}  \right], 
    B^k = \left[ \begin{array}{ccc}
    -\nabla g(x^k)  &  0 & 0  \\
    0  &  0 &  0 \\
    \end{array}  \right], \\
    & C^k = \left[ \begin{array}{cc}
    (\nabla g(x^k))^\T    &  0  \\
    0  &  0  \\
    0  &    0 
    \end{array}  \right], \ \text{and} \
    D^k = \left[ \begin{array}{ccc}
    0  & -I  & 0  \\
    I &  0  & -I  \\
    0 &   \D(z^k)  & \D(s^k)    
    \end{array}  \right].
    \end{align*}
    From Lemma \ref{boundAboveBelow},  the two sequences $\{s^k\}$ and $\{z^k\}$ are
    bounded and each component of the two sequences are bounded below and away from zeros, 
    therefore, the sequence $\{(D^k)^{-1}\}$ is also bounded, where
    \[
    (D^k)^{-1} = \left[ \begin{array}{ccc}
    \D(s^k)^{-1}\D(z^k) &  I  & \D(s^k)^{-1}  \\
    -I &  0  & 0  \\
    \D(s^k)^{-1}\D(z^k) &  0  &  \D(s^k)^{-1}
    \end{array}  \right].
    \]
    We know that 
    $\nabla_x^2 L(v^k) + \nabla g(x^k) \D(s^k)^{-1} \D(z^k) \nabla g(x^k)^{\T}$ is invertible from Lemma~\ref{boundAboveBelow} and (B3),
    therefore, 
    $(\nabla h(x^k))^{\T} \left( \nabla_x^2 L(v^k) + \nabla g(x^k) \D(s^k)^{-1} \D(z^k) (\nabla g(x^k))^{\T}
    \right)^{-1} \nabla h(x^k)$ is also invertible from (B1).
    Therefore, 
    \[
    H^k :=  A^k - B^k(D^k)^{-1}C^k   = \left[ \begin{array}{cc}
    \nabla_x^2 L(v^k) + \nabla g(x^k) \D(s^k)^{-1} \D(z^k) \nabla g(x^k)^{\T}  & \nabla h(x^k) \\
    (\nabla h(x^k))^{\T}  &  0
    \end{array}  \right]
    \]
    is invertible from Lemma \ref{blockInverse}.
    Since 
    $A^k$ and $H^k$ are invertible, we again use Lemma \ref{blockInverse} 
    to show that 
    $F'(v^k)$
    is invertible. 
    
    Next, we show the boundedness 
    of $\{[F'(v^k)]^{-1}\}$.
    Since $[F'(v^k)]^{-1}$ is given by
    \[
    [F'(v^k)]^{-1} = \left[\begin{array}{cc}
    (H^k)^{-1} & -(H^k)^{-1} B^k (D^k)^{-1} \\ 
    -(D^k)^{-1} C^k (H^k)^{-1} & (D^k)^{-1} C^k (H^k)^{-1} B^k (D^k)^{-1} + (D^k)^{-1}
    \end{array}\right],
    \]
    we need to show $\{(H^k)^{-1}\}$ is bounded,
    For each $k$, $(H^k)^{-1}$ is given as follows:
    \[
    (H^k)^{-1} = \left[\begin{array}{cc}
    \bar{L}^{-1} - \bar{L}^{-1} \nabla h(x^k) \bar{H}^{-1} (\nabla h(x^k))^{\T} \bar{L}^{-1} & \bar{L}^{-1} \nabla h(x^k) \bar{H}^{-1} \\
    \bar{H}^{-1} (\nabla h(x^k))^{\T} \bar{L}^{-1} & - \bar{H}^{-1}
    \end{array}\right],
    \]
    where $\bar{L} = 
    \nabla_x^2 L(v^k) + \nabla g(x^k) \D(s^k)^{-1} \D(z^k) (\nabla g(x^k))^{\T}$ and $\bar{H} = (\nabla h(x^k))^{\T} \bar{L}^{-1} \nabla h(x^k)$.
    Therefore, it is enough to show 
    the boundedness of $\bar{L}$ and $\bar{H}$, and this is done by
    Assumptions (B4) and (B3), and Lemma~\ref{boundAboveBelow}.
    This completes the proof.
    \hfill\qed
\end{proof}

The following lemma follows directly from Lemma~\ref{invertibility}.

\begin{lemma}\label{ettzTheorem}
    Assume that (B1)-(B4) hold. If $\{v_k\} \subset \Omega(\epsilon)$, 
    then (i) Steps 3 and 5 in Algorithm \ref{mainAlgo1} are well-defined,
    and (ii) the sequences $\{ \dot{v}^k \}$ and $\{ \ddot{v}^k \}$ 
    are bounded.
\end{lemma}
\begin{proof}
    The claim (i) follows directly from Lemma~\ref{invertibility}.
    In the view of (\ref{firstOrder}), the boundedness of 
    $\{ [F'(v^k)]^{-1} \}$ and $\{ v^k \}$ guarantees that of 
    $\{ \dot{v}^k \}$. Using 
    (\ref{secondOrder}), the boundedness of $\{ \ddot{v}^k \}$ 
    can be shown from a similar argument.
    \hfill\qed
\end{proof}

These lemmas allow us to show that $\{\tilde{\alpha}_k \}$ 
is bounded below and away from zero.

\begin{lemma}
Assume that (B1)-(B4) hold. 
If $\{v^k\} \subset \Omega(\epsilon)$, then the sequence
$\{ \tilde{\alpha}_k \}$ is bounded below and away from zero.
\label{barAlpha}
\end{lemma}
\begin{proof}
We can rewrite (\ref{key1a}) as
\begin{equation}
(1 - \delta) w^k +\dot{w}^k\sin(\alpha)+\ddot{w}^k(1-\cos(\alpha))
\ge 0.
\label{alphaiNew}
\end{equation}
From Lemma \ref{boundAboveBelow},
 $\{(w^k, s^k, z^k)\} \subset \R_{++}^{3p}$ 
is bounded below and away from zero, thus
 $\{(1 - \delta) w^k\} $ is bounded below and away from zero.
Since $\{\dot{w}^k\}$ and $\{\ddot{w}^k\}$ are bounded from Lemma~\ref{ettzTheorem}, 
$\{\tilde{\alpha}_k\}$ should be bounded below and away from zero such that the
inequality (\ref{alphaiNew}) holds for all $\alpha \in [0, \tilde{\alpha}_k]$. 
We can apply the same arguments
to $\{s^k \}$ and $\{z^k \}$. This proves the Lemma.
\hfill \qed
\end{proof}

Next, we show that  $\{ \check{\alpha}_k \}$ is bounded below and away from zero.
\begin{lemma}
If $\{v^k\} \subset \Omega(\epsilon)$, then the sequence
$\{ \check{\alpha}_k \}$ is bounded below and away from zero.
\label{checkAlpha}
\end{lemma}
\begin{proof}
Since $\{v^k\} \subset \Omega(\epsilon)$, we have 
$\epsilon \le \phi(v^k)$. 
From (\ref{objReduction}),
it follows that 
$\nabla_{\alpha} \phi({v}\AN{\alpha})|_{\alpha=0} \ge  
2\phi({v}) (1-\sigma) \ge  2\epsilon(1-\sigma) 
\ge 2 \epsilon (1-\bar{\sigma})$. 
Therefore, if $\sin(\alpha) \ge \frac{\delta}{2\beta \epsilon(1-\bar{\sigma})}$,
we have $\beta \sin(\alpha) \nabla_{\alpha} \phi({v}\AN{\alpha})|_{\alpha=0}
> \delta$. From \eqref{alphaCheck}, we can take $\check{\alpha}_k \ge 
\sin^{-1}\left(\frac{\delta}{2\beta \epsilon(1-\bar{\sigma})}\right)$,
and this implies $\{\check{\alpha}_k\}$ is bounded below and away from zero.

\hfill \qed
\end{proof}

Finally, we  show that $\{ \hat{\alpha}_k \}$ is bounded below and away from zero in Lemma~\ref{hatAlpha} using 
a formula related to the arc of ellipse ${\cal E}$.
\begin{lemma}
    Assume that ${v}$ is the current point (\textit{i.e.}, $v = v^k$ at the $k$th iteration) and $\dot{v}$ and $\ddot{v}$ satisfy 
    (\ref{firstOrder}) and (\ref{secondOrder}). Let $v\AN{\alpha}$ be computed with (\ref{vAlpha}). Then,
    \begin{eqnarray}
    {z}_i\AN{\alpha} {s}_i\AN{\alpha} & = &  {z}_i  {s}_i (1-\sin(\alpha)) +\sigma  {\mu} \sin(\alpha)
    -(\dot{z}_i\ddot{s}_i+\ddot{z}_i\dot{s}_i)\sin(\alpha)(1-\cos(\alpha))
    \nonumber \\
    & &  +(\ddot{z}_i\ddot{s}_i-\dot{z}_i\dot{s}_i) (1-\cos(\alpha))^2.
    \label{compWise}
    \end{eqnarray}
\end{lemma}
\begin{proof}
    Using the last rows of (\ref{firstOrder}) and (\ref{secondOrder}), 
    we have
    \begin{eqnarray}
    {z}_i\AN{\alpha}{s}_i\AN{\alpha} & = & 
    [{z}_i-\dot{z}_i\sin(\alpha)+\ddot{z}_i(1-\cos(\alpha))]
    [{s}_i-\dot{s}_i\sin(\alpha)+\ddot{s}_i(1-\cos(\alpha))]
    \nonumber \\
    & = & {z}_i {s}_i-(\dot{z}_i{s}_i+ {z}_i\dot{s}_i)\sin(\alpha)
    +(\ddot{z}_i {s}_i+ {z}_i\ddot{s}_i) (1-\cos(\alpha)) 
    + \dot{z}_i\dot{s}_i \sin^2(\alpha) \nonumber \\
    & & -(\dot{z}_i\ddot{s}_i+\ddot{z}_i\dot{s}_i)
    \sin(\alpha)(1-\cos(\alpha))
    +\ddot{z}_i\ddot{s}_i  (1-\cos(\alpha))^2
    \nonumber \\
    & = & {z}_i {s}_i (1-\sin(\alpha)) +\sigma {\mu} \sin(\alpha) -2\dot{z}_i\dot{s}_i
    (1-\cos(\alpha)) + \dot{z}_i\dot{s}_i \sin^2(\alpha) \nonumber \\
    & & -(\dot{z}_i\ddot{s}_i+\ddot{z}_i\dot{s}_i)
    \sin(\alpha)(1-\cos(\alpha))
    +\ddot{z}_i\ddot{s}_i (1-\cos(\alpha))^2
    \nonumber \\
    & = & {z}_i {s}_i (1-\sin(\alpha)) +\sigma {\mu} \sin(\alpha)
    +\dot{z}_i\dot{s}_i (\sin^2(\alpha)+2\cos(\alpha)-2)
    \nonumber \\
    & & -(\dot{z}_i\ddot{s}_i+\ddot{z}_i\dot{s}_i)\sin(\alpha)(1-\cos(\alpha))
    +\ddot{z}_i\ddot{s}_i (1-\cos(\alpha))^2.
    \nonumber 
    \end{eqnarray}
    Substituting $\sin^2(\alpha)+2\cos(\alpha)-2=-1+2\cos(\alpha)
    -\cos^2(\alpha)=-(1-\cos(\alpha))^2$ into the last equation gives
    (\ref{compWise}).
    \hfill \qed 
\end{proof}

\begin{lemma} 
Assume that (B1)-(B4) hold. If $\{v^k\} \subset \Omega(\epsilon)$ 
for some $\epsilon > 0$, then $\{\hat{\alpha}_k\}$ is bounded below and away from zero.
\label{hatAlpha}
\end{lemma}
\begin{proof}  

For each $k$, 
find $i$ such that  $z_i^1 s_i^1 = \min(\D(z^1) s^1)$, and let 
$\eta_1^k=  \dot{z}_i^k\ddot{s}_i^k+\ddot{z}_i^k\dot{s}_i^k$ and
$\eta_2^k= \ddot{z}_i^k\ddot{s}_i^k-\dot{z}_i^k\dot{s}_i^k$.
Since $\{\dot{v}^k\}$ and $\{\ddot{v}^k\}$ are bounded
due to Lemma~\ref{ettzTheorem},
the sequences $\{ | \eta_1^k | \}$ and $\{ | \eta_2^k | \}$
are also bounded.

The proof is based on induction.
For $k=1$, from (\ref{compWise}) and (\ref{conC1}), we have 
\begin{eqnarray}
& & \min (\D(z^1) s^1) - \frac{1}{2} \min (\D(z^0) s^0) \frac{\phi (v^1)}{\phi (v^0)}  \nonumber \\
& \ge & z_i^1 s_i^1 - \frac{1}{2} \min (z^0 s^0) 
[1- 2\beta (1-\sigma_0) \sin(\alpha_0)]  \nonumber \\
& \ge &  {z}_i^0  {s}_i^0 (1-\sin(\alpha_0)) +\sigma_0  {\mu_0} \sin(\alpha)
-\eta_1^0\sin(\alpha_0)(1-\cos(\alpha_0))
+\eta_2^0 (1-\cos(\alpha_0))^2  \nonumber \\
& & -\frac{1}{2} (z_i^0 s_i^0) [1- 2\beta (1-\sigma_0) \sin(\alpha_0)] \nonumber \\
& \ge & \frac{1}{2} {z}_i^0  {s}_i^0 - {z}_i^0  {s}_i^0\sin(\alpha_0)
+\sigma_0  {\mu_0} \sin(\alpha_0) -\eta_1^0\sin(\alpha_0)(1-\cos(\alpha_0)) \nonumber \\
& & +\eta_2^0 (1-\cos(\alpha_0))^2
+{z}_i^0  {s}_i^0 \beta (1-\sigma_0) \sin(\alpha_0).
\label{inter1}
\end{eqnarray}

Since $v^k \in \Omega(\epsilon)$, we know 
$z_i^k s_i^k \ge \frac{1}{2} \min(\D(z^0) s^0)
\frac{\phi(v^k)}{\phi(v^0)} \ge \frac{1}{2} \min(\D(z^0) s^0)
\frac{\epsilon}{\phi(v^0)} > 0$, therefore
there must be $\alpha_0 > 0 $ such that
the last express in (\ref{inter1}) is greater than zero.
Next, for $k>1$, assume that there exists $\alpha_{k-1} >0$ such that 
\begin{equation}
\min (\D(z^k) s^k) - \frac{1}{2} \min (\D(z^0) s^0) 
\frac{\phi (v^k)}{\phi (v^0)} > 0,
\label{inter2}
\end{equation}
then we show that there exists $\alpha_{k}>0$ such that 
\begin{equation}
\min (\D(z^{k+1}) s^{k+1}) - \frac{1}{2} \min (\D(z^0) s^0) 
\frac{\phi (v^{k+1})}{\phi (v^0)} > 0.
\nonumber
\end{equation}

From (\ref{compWise}) and (\ref{conC1}), it holds that 
\begin{eqnarray}
& & \min (\D(z^{k+1}) s^{k+1}) - \frac{1}{2} \min (\D(z^0) s^0) 
\frac{\phi (v^{k+1})}{\phi (v^0)}  \nonumber \\
& \ge & z_i^{k+1} s_i^{k+1} - \frac{1}{2} \min (\D(z^0) s_j^0) 
\frac{\phi (v^k)}{\phi (v^0)}[1- 2\beta (1-\sigma_k) \sin(\alpha_k)]  
  \nonumber \\
& \ge &  {z}_i^k  {s}_i^k (1-\sin(\alpha_k)) +\sigma_k  {\mu_k} \sin(\alpha_k)
-\eta_1^k\sin(\alpha_k)(1-\cos(\alpha_k))
+\eta_2^k (1-\cos(\alpha_k))^2  \nonumber \\
& & -\frac{1}{2} \min (\D(z^0) s^0)  
\frac{\phi (v^k)}{\phi (v^0)}[1- 2\beta (1-\sigma_k) \sin(\alpha_k)] 
\label{inter3}
\end{eqnarray}
Since ${z}_i^k  {s}_i^k \ge \min (\D(z^k) s^k) \ge \frac{1}{2} \min(\D(z^0) s^0)
\frac{\epsilon}{\phi(v^0)} > 0 $ and (\ref{inter2}), 
we can find $\alpha_k >0$ such
that the last express in (\ref{inter3}) is greater than zero.

We already know that  $\{z_i^k s_i^k\}$ and $\{\sigma_k\}$ are bounded below and away from zero, and $\{ |\eta_1^k|\}$ and $\{|\eta_2^k|\}$ are bounded 
due to Lemma~\ref{ettzTheorem}. Therefore, $\{\hat{\alpha}_k\}$  
is bounded below and away from zero.
\hfill \qed
\end{proof}

We are now ready to prove the convergence of Algorithm~\ref{mainAlgo1}.
From Lemmas~\ref{barAlpha}, \ref{checkAlpha}, \ref{hatAlpha},  
we already establish that $\{\alpha_k\}$ is bounded below and away from zero.

\begin{theorem}\label{global} 	
Assume (B1)-(B4) hold and $\phi(v^0)$ is bounded. Then, 
	(i) for all $k \ge 0$, the sequence $\{\phi(v^k)\}$ 
decreases in a constant rate, and 
	(ii) the algorithm terminates in finite iterations and
the finds an $\epsilon$-approximate solution of the 
problem \eqref{NP}.
\end{theorem}
\begin{proof}
Since $\{\alpha_k\}$ is bounded below and away from zero,
there must be $\bar{\alpha} > 0$ such that 
$\check{\alpha}_k \ge \alpha_k \ge \bar{\alpha} > 0$.
This shows that $\{\phi(v^k)\}$ decreases in a constant rate 
due to \eqref{conC1}.
Since $\phi(v^0)$ is bounded, and $\{\phi(v^k)\}$ decreases 
in a constant rate, it needs only a finite iterations $K$
to $\phi(v^K) \le \epsilon$
with $(w^K, s^K, z^K) \in \R_+^{3p}$.
\hfill \qed
\end{proof}

\begin{remark}
Although assumptions similar to \cite{ettz96} are made
in (B1)-(B4), we obtained a stronger finite convergence
result than \cite{ettz96}.
\end{remark}

\section{Numerical Experiments}\label{sec:experiments}

We conducted numerical experiments to compare
the performance of the proposed arc-search algorithm (Algorithm~\ref{mainAlgo1}) and a line-search algorithm.
A framework of the line-search algorithm we used 
in the numerical experiments is given as follows.
The main difference from Algorithm~\ref{mainAlgo1}
is that Algorithm~\ref{mainAlgo_line} uses only $\dot{v}$ and not $\ddot{v}$.
\begin{algorithm} 	\label{mainAlgo_line}
{\bf (an infeasible line-search type interior-point algorithm for nonlinear 
programming problems)} \newline
	\indent Parameters: $\epsilon>0$, $\delta>0$,  
$\beta \in (0, \frac{1}{2}]$, and $\gamma_{-1}=1$.
\newline\indent Initial point: $v^0 = (x^0,y^0,w^0,s^0,z^0)$ 
such that $(w^0,s^0,z^0)>0$ and $w^0=z^0$.
\newline
\newline
\indent {\bf for} iteration $k=0,1,2,\ldots$
\begin{itemize}
	\item[] Step 1:  If $\phi(v^k) \le \epsilon$, stop.
	\item[] Step 2: Calculate $\nabla_x L(v^k)$, $h(x^k)$, $g(x^k)$, 
	$\nabla_x^2 L(v^k)$, $\nabla_x h(x^k)$, and $\nabla_x g(x^k)$.
	\item[] Step 3: Select $\sigma_k$ such that 
    $\bar{\sigma} \le \sigma_k < \frac{1}{2}$
	and let  
	$\dot{v}^k = (\dot{x}^k,\dot{y}^k,\dot{w}^k,\dot{s}^k,\dot{z}^k)$
	of the solution of (\ref{firstOrder}) with ${v} = v^k$.
	\item[] Step 4:  Choose $\gamma_k$ such that $\frac{1}{2} \le \gamma_{k} 
	\le \gamma_{k-1}$, and find  appropriate 
	$\alpha_k > 0$ using $\gamma_k$
     such that $w^{k+1} \in \R_{++}^p, s^{k+1} \in \R_{++}^p$
    and $\phi(v^{k+1}) < \phi(v^k)$ hold.
	\item[] Step 5:  Update $v^{k+1} =  v^k  + \alpha_k \dot{v}^k$.
\end{itemize}
\indent\indent {\bf end (for)} 
	\hfill \qed
\end{algorithm}

A main objective of the numerical experiments in this paper is to observe numerical behaviors of the arc-search algorithm (Algorithm~\ref{algorithm:arc-search}) compared with 
the line-search algorithm (Algorithm~\ref{mainAlgo_line})
Existing packages often employ many techniques 
to improve numerical stability or computation time.
However, such techniques might prevent us from focusing the difference of two algorithms
and 
implementing such techniques should be separated as a future work,
therefore, we did not include existing packages in the numerical experiments.

For the test problems, we used the CUTEst test set \cite{gould2015cutest}.  According to the types of problems, we classified the entire set 
into four types; LP (linear programming) problems, QP (quadratic programming) problems,  QCQP (quadratically-constrained quadratic programming) problems and Others. 
Here, the problems in ``Others''  include a function whose degree is higher than 2.
In the numerical experiments, we excluded LP and QP types,
since the proposed arc-search algorithm in this paper is designed for NLPs,
and existing arc-search algorithms \cite{yang13, yang11, yy18}  
proposed for LP and QP types are more effective for these types.
The variable size $n$ in QCQP and Others ranges from $2$ to $2002$, and the total number of constraints in $h,g$  from $2$ to $1722$. 

The commands of the CUTEst provides the gradient vectors
and the Hessian matrices, but not the third derivatives.
Therefore, we used numerical differentiation for computing $\nabla_x^3 L(v) $, for example, we computed
\begin{equation}
\nabla_{x_i}(\nabla_x^2 L(x, y, w,s,z)) = \frac{\nabla_x^2 L(x+\hat{\epsilon} e_i, y, w,s,z) - \nabla_x^2 L(x, y, w,s,z)}{\hat{\epsilon}} \label{eq:num-diff}
\end{equation}
where $e_i$ is the $i$th unit vector and $\hat{\epsilon}$ is a small positive number.
In the numerical experiments, we set $\hat{\epsilon} = 10^{-4}$.

For the parameters, we set 
$\delta = 10^{-3}$ and 
$\gamma_k = \frac{1}{2}, \sigma_k = \frac{1}{8}\min\{1, \phi(v^k)p/(\mu^k)^2\}$ for 
all $k$.
We stop the algorithms when the deviation from the KKT conditions gets 
smaller than a tolerance, $\phi(v^k) \le 10^{-8}$, 
or the iteration number exceeds a limit, $k \ge 1000$. 

\subsection{Numerical Results}\label{sec:results}

We compare the number of iterations and the computation time 
with 
problems that are solved by
both Algorithm~\ref{mainAlgo1} and Algorithm~\ref{mainAlgo_line},
The detailed tables of the numerical results are put in \ref{sec:detailedResults}.
For summarizing the numerical results, we utilize 
the performance profiling proposed in \cite{gould2016note}.
In the performance profiling for the computation time, 
the vertical axis $P(r_{p,s} \le \tau)$ is the proportion of the problems 
in the numerical experiments
for which $r_{p,s}$ is at most $\tau$, where $r_{p,s}$  is the ratio of the computation time
of the algorithm against
the shorter computation time among the two algorithms.
Simply speaking, the algorithm that approaches to 1 at smaller $\tau$ is better.

Figure~\ref{10-01_pp} shows the performance profile
of Algorithm~\ref{mainAlgo1} and Algorithm~\ref{mainAlgo_line}.
We observe that the number of iterations is less than that of the line-search algorithm.
We can consider that the proposed arc-search algorithm approximates the central path better than the line-search algorithm.
In contrast, in the viewpoint of the computation time, the proposed arc-search algorithm
consumed a longer time.
We found that the main bottleneck in Algorithm~\ref{mainAlgo1}
was the right-hand side of (\ref{secondOrder}), in particular, the computation on 
$ (\nabla_x^3 L(v))\dot{x} \dot{x}$,
$ (\nabla_x^2 h(x))\dot{y} \dot{x}$,
$ (\nabla_g^2 h(x))\dot{z} \dot{x}$,
$ (\nabla_x^2 h(x))\dot{x} \dot{x}$,
and      
$(\nabla_x^2 g(x))\dot{z} \dot{x})$.
We will discuss these higher-order derivatives in Section~\ref{sec:discussions}.

\begin{figure}[htb]
     \begin{minipage}{0.48\textwidth}
      \iffigure
      \iffigureeps
       \includegraphics[width=7cm]{iterAll_10-01.eps}
       \else
       \includegraphics[width=7cm]{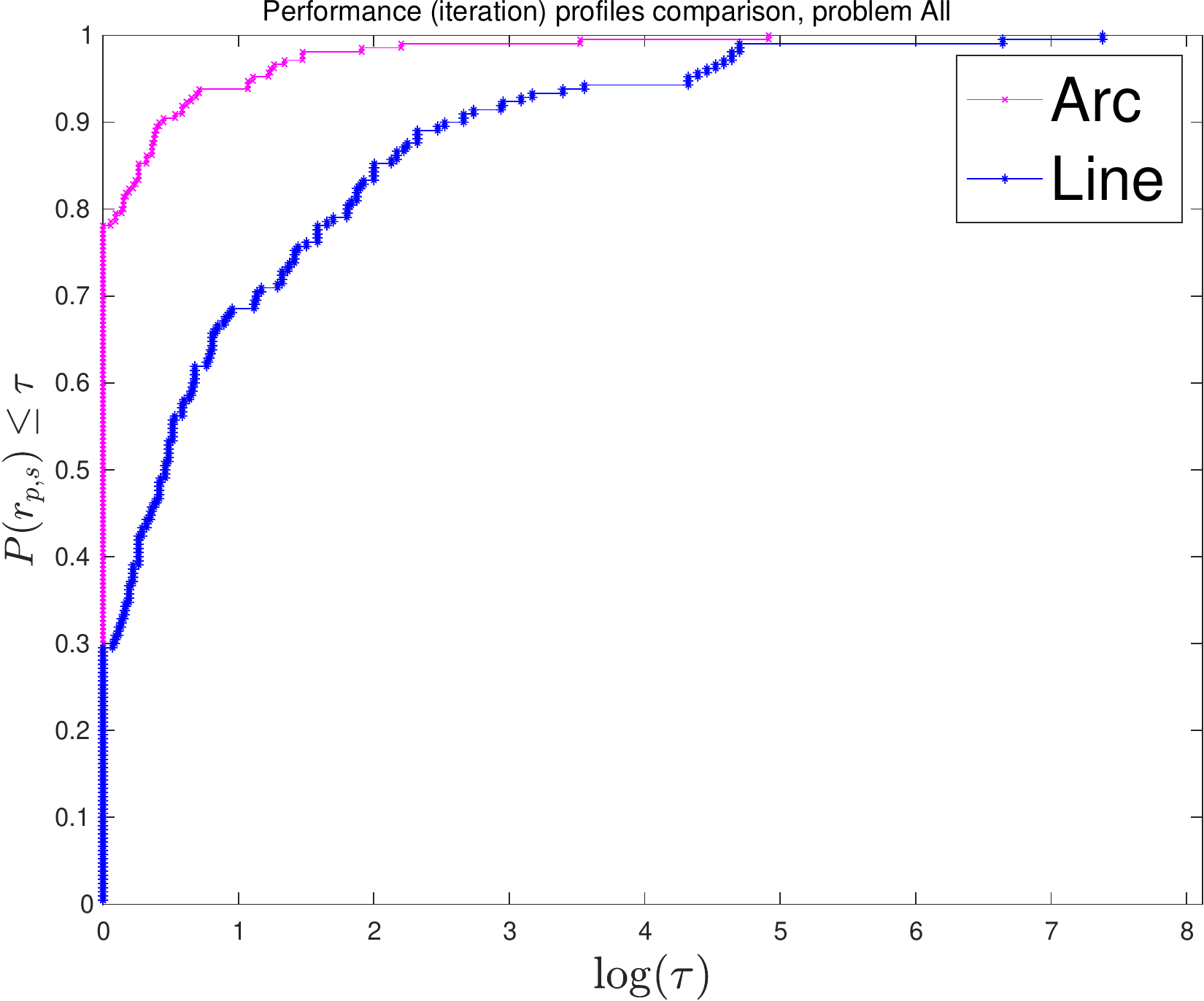}
       \fi
        \fi
        \end{minipage}
	\begin{minipage}{0.48\textwidth}
    \iffigure
      \iffigureeps
    \includegraphics[width=7cm]{time_AllAll_10-01.eps}
    \else
    \includegraphics[width=7cm]{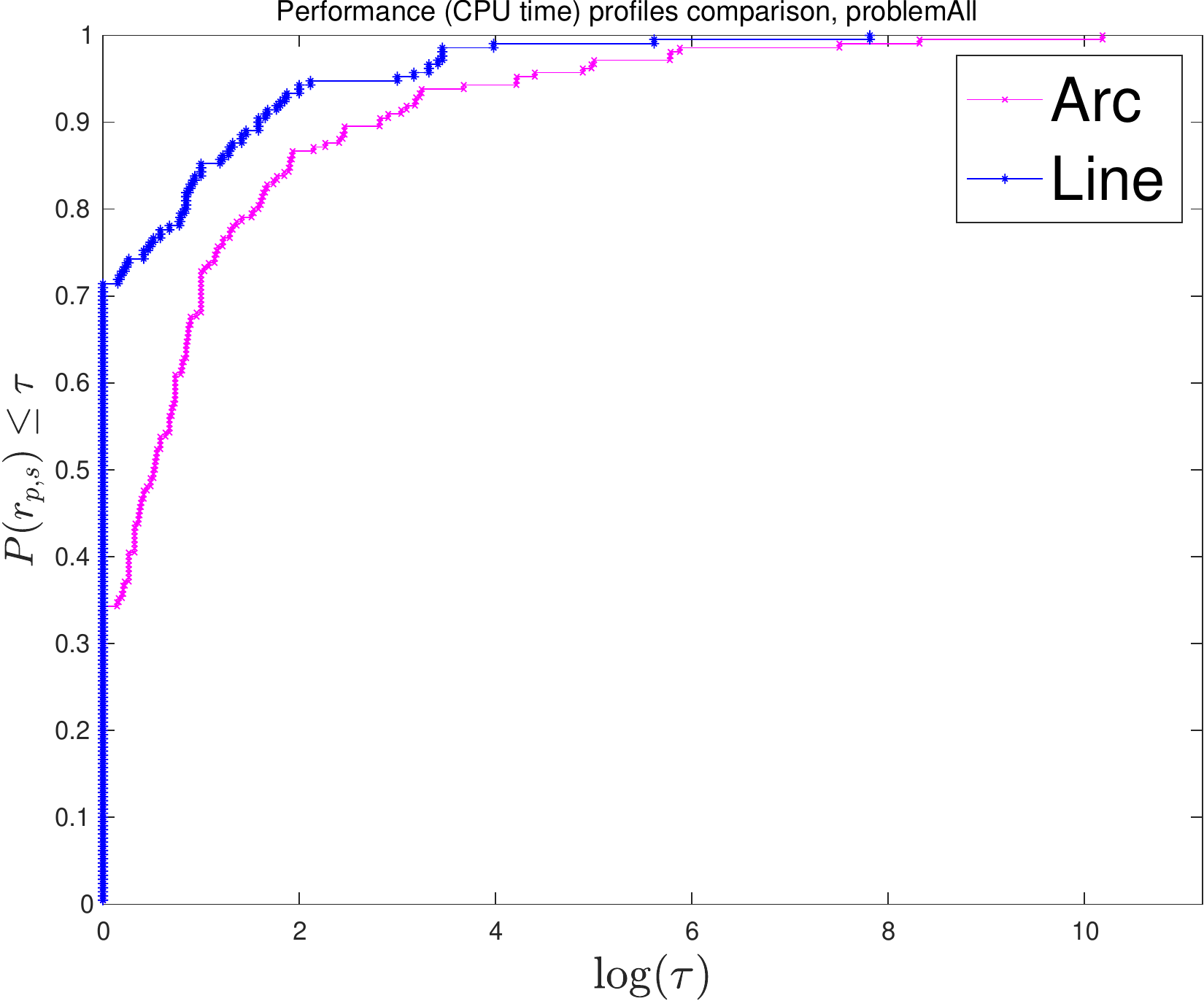}
    \fi
    \fi
    \end{minipage}
    \caption{Performance profiles of  the number of iterations (left) and the computation time (right) for all solvable problems.} 
    \label{10-01_pp}
\end{figure}

Figures~\ref{10-01_qcqp} illustrates
the performance profile for QCQPs.
This result indicates that the computation time of the proposed 
arc-search algorithm is competitive with the line-search algorithm
in QCQPs.
The degrees of the functions in QCQPs are at most 2,
therefore, the approximation with the ellipse fits the central path well
and the number of iterations is much smaller than the line-search algorithm.

\begin{figure}[htb]
    \begin{minipage}{0.48\textwidth}
        \iffigure
              \iffigureeps        
         \includegraphics[width=7cm]{iterQCQP_10-01.eps}
         \else
        \includegraphics[width=7cm]{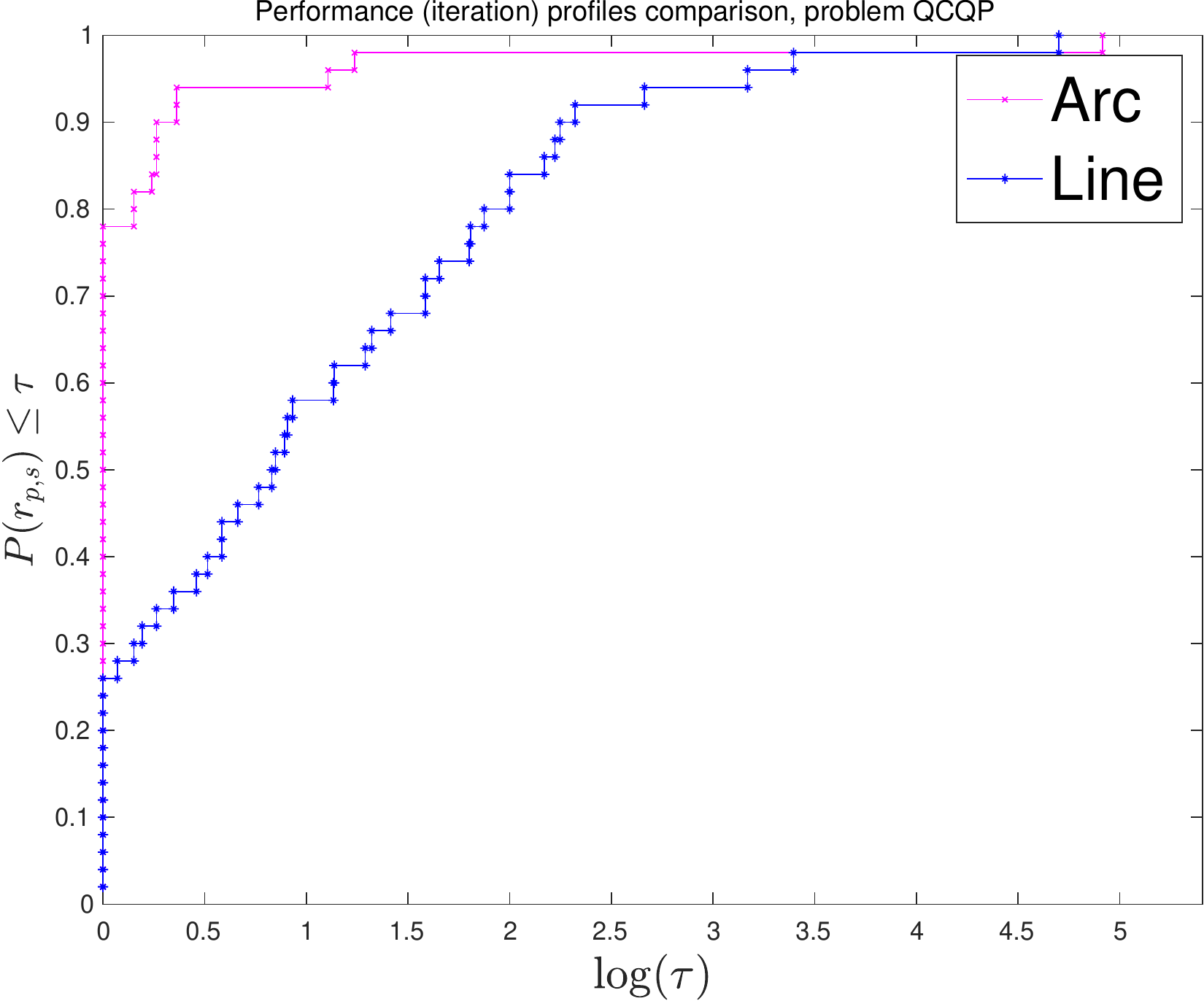}
        \fi
        \fi
    \end{minipage}
    \begin{minipage}{0.48\textwidth}
        \iffigure
              \iffigureeps        
        \includegraphics[width=7cm]{time_AllQCQP_10-01.eps}
        \else
        \includegraphics[width=7cm]{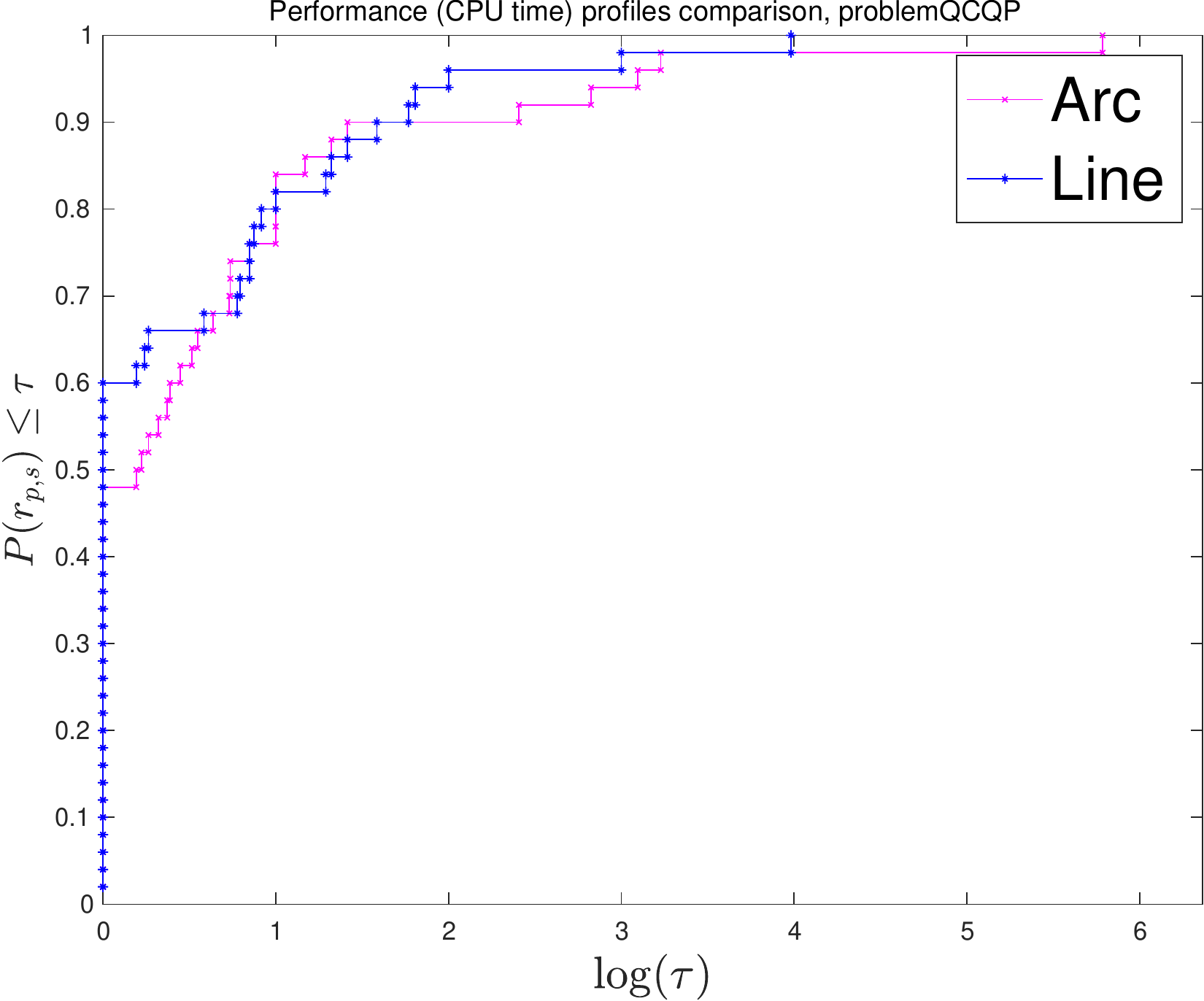}
        \fi
        \fi
    \end{minipage}
    \caption{Performance profiles of  the number of iterations (left) and the computation time (right) for QCQP problems} \label{10-01_qcqp}
\end{figure}

\subsection{High-order derivatives}\label{sec:discussions}
As pointed out above, 
the main bottleneck of the proposed arc-search algorithm
is the computation of the high-order derivatives;
$ (\nabla_x^3 L(v))\dot{x} \dot{x}$,
$ (\nabla_x^2 h(x))\dot{y} \dot{x}$,
$ (\nabla_g^2 h(x))\dot{z} \dot{x}$,
$ (\nabla_x^2 h(x))\dot{x} \dot{x}$,
and $(\nabla_x^2 g(x))\dot{z} \dot{x})$.
However, these  higher-order derivatives appear only in 
the right-hand side of (\ref{secondOrder}) for obtaining $\ddot{v}$.
Since the second-order approximation $\ddot{v}$ gives a less influence on  ${v}\AN{\alpha}$
than the first-order approximation $\dot{v}$  when $\alpha$ is small,
we can expect that small deviations in the computation of $\ddot{v}$ would not affect the approximation of ${v}$ so much. In addition, we can remove the effect of numerical errors in the numerical differentiations like (\ref{eq:num-diff}).
Based on these intuitions,
we examine another approximation with $\ddot{\ddot{v}} = (\ddot{\ddot{x}}, \ddot{\ddot{y}}, \ddot{\ddot{w}}, \ddot{\ddot{s}}, \ddot{\ddot{z}})$ 
defined as the solution of the following system in which we ignored
the higher-order derivatives of (\ref{secondOrder}):
{\small
	\begin{eqnarray}
	\hspace{-2.2cm}
	\left[ \begin{array}{ccccc}
	\nabla_x^2 L(v) & \nabla h(x) & -\nabla g(x) & 0  & 0 \\
	\left( \nabla h(x)\right)^{\T} & 0 & 0 & 0 & 0  \\
	\left( \nabla g(x)\right)^{\T} & 0 & 0 & -I  & 0 \\
	0  & 0  &  I  &  0  &  -I  \\
	0 & 0 & 0  & \D({z}) & \D({s})
	\end{array} \right]
	\left[ \begin{array}{c}
	\ddot{\ddot{x}} \\ \ddot{\ddot{y}} \\ \ddot{\ddot{w}} \\ \ddot{\ddot{s}} \\ \ddot{\ddot{z}}
	\end{array} \right]
	& = & \left[ \begin{array}{l}
0 \\
0 \\
0 \\
0  \\
	-2 \D(\dot{z}) \dot{s} 
	\end{array} \right].
	\nonumber	\label{ddotddotv}
	\end{eqnarray}
}

Figure~\ref{19-04_pp} compares the arc-search algorithm with $\ddot{\ddot{v}}$
and the line-search algorithm (Algorithm~\ref{mainAlgo_line}) using the performance profiling.
In the viewpoint of the number of iterations, the arc-search algorithm keeps its superiority.
In addition, the arc-search algorithm solves the problems in a shorter time than the line-search algorithm, since we skip the main bottlenecks.

\begin{figure}[htb]
    \begin{minipage}{0.48\textwidth}
        \iffigure
              \iffigureeps        
        \includegraphics[width=7cm]{iterAll_19-04.eps}
        \else
        \includegraphics[width=7cm]{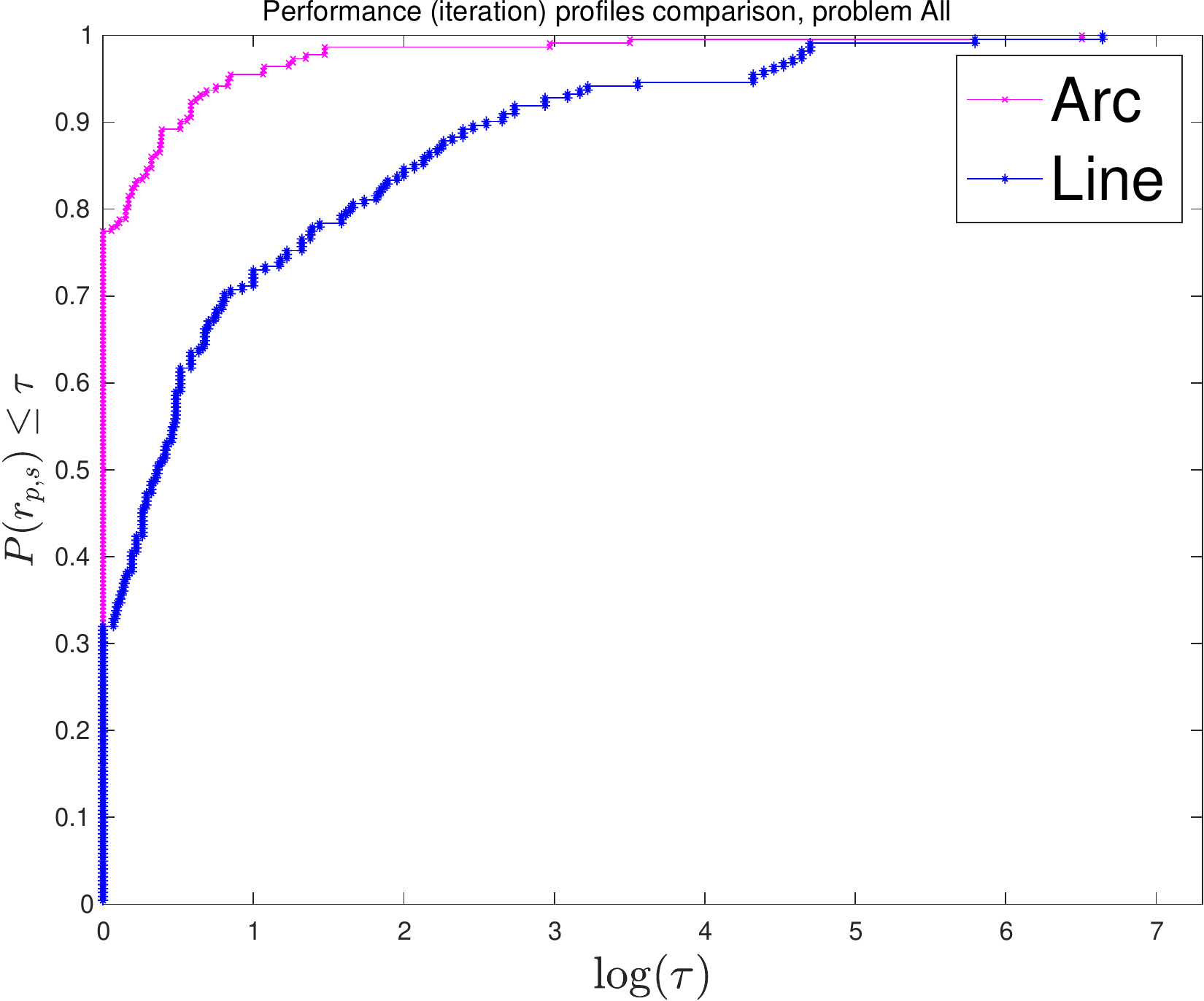}
        \fi
        \fi
    \end{minipage}
    \begin{minipage}{0.48\textwidth}
        \iffigure
              \iffigureeps        
         \includegraphics[width=7cm]{time_AllAll_19-04.eps}
         \else
        \includegraphics[width=7cm]{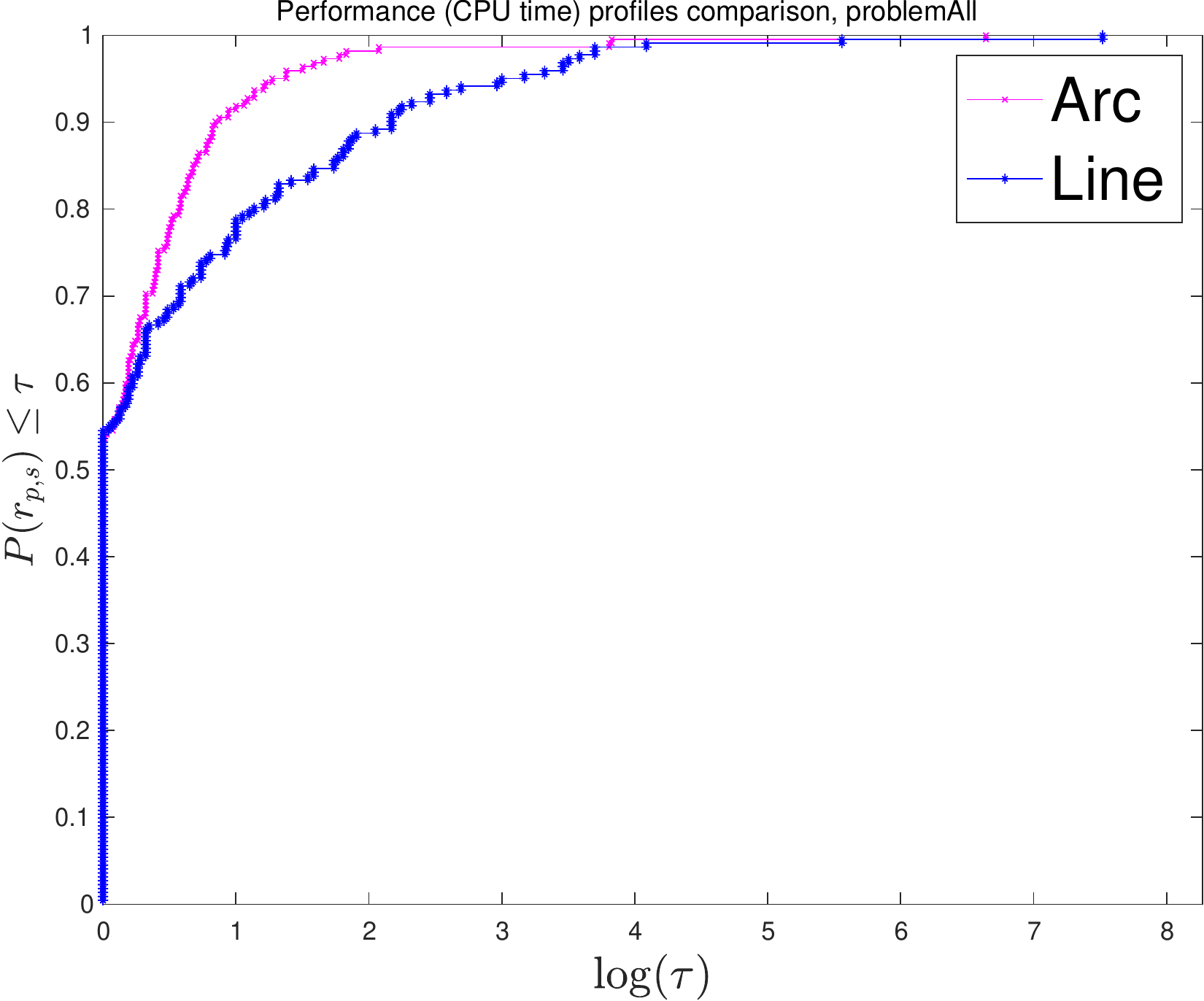}
        \fi
        \fi
    \end{minipage}
    \caption{Performance profiles of  the number of iterations (left) and the computation time (right) with the use of $\ddot{\ddot{v}}$} \label{19-04_pp}
\end{figure}

Since $\ddot{\ddot{v}}$ can not draw the ellipse ${\cal E}$ exactly, we cannot apply the same theoretical developments in the previous section. However, these numerical results give promising insights for  further improvements on the arc-search algorithm.

\section{Conclusions}\label{sec:conclusion}
In this paper, we extend the arc-search algorithm, which approximates the central path with an arc of the ellipse, for NLPs and also discuss the convergence of the proposed algorithm.
From the results of numerical experiments, the arc-search algorithm succeeded in reducing the number of iterations compared with the line-search algorithm.

As a future work, we should focus the computation time reduction of the arc-search algorithm.
In particular, we expect the drop of the high-order derivatives in the computation of $\ddot{v}$ will bring us an enhancement of the algorithm as observed in Section~\ref{sec:discussions},
though the deviation from the arc due to the drop should be theoretically addressed.
We should also incorporate some implementation techniques to improve the numerical stability for NLPs.

\bibliographystyle{siam}     
\bibliography{scholar}

\begin{thebibliography}{10}

\bibitem{bgn00}
{\sc R.~H. Byrd, J.~C. Gilbert, and J.~Nocedal}, {\em A trust region method
  based on interior point techniques for nonlinear programming}, Mathematical
  Programming, 89 (2000), pp.~149--185.

\bibitem{bhn99}
{\sc R.~H. Byrd, M.~E. Hribar, and J.~Nocedal}, {\em An interior point
  algorithm for large-scale nonlinear programming}, SIAM Journal on
  Optimization, 9 (1999), pp.~877--900.

\bibitem{ettz96}
{\sc A.~S. El-Bakry, R.~A. Tapia, T.~Tsuchiya, and Y.~Zhang}, {\em On the
  formulation and theory of the {Newton} interior-point method for nonlinear
  programming}, Journal of Optimization Theory and Applications, 89 (1996),
  pp.~507--541.

\bibitem{fg98}
{\sc A.~Forsgren and P.~E. Gill}, {\em Primal-dual interior methods for
  nonconvex nonlinear programming}, SIAM Journal on Optimization, 8 (1998),
  pp.~1132--1152.

\bibitem{gould2016note}
{\sc N.~Gould and J.~Scott}, {\em A note on performance profiles for
  benchmarking software}, ACM Transactions on Mathematical Software, 43 (2016),
  p.~15.

\bibitem{gould2015cutest}
{\sc N.~I. Gould, D.~Orban, and P.~L. Toint}, {\em {CUTEst}: a constrained and
  unconstrained testing environment with safe threads for mathematical
  optimization}, Computational Optimization and Applications, 60 (2015),
  pp.~545--557.

\bibitem{kheirfam17}
{\sc B.~Kheirfam}, {\em An arc-search infeasible interior-point algorithm for
  horizontal linear complementarity problem in the {$N^{-\infty}$}
  neighbourhood of the central path}, International Journal of Computer
  Mathematics, 94 (2017), pp.~2271--2282.

\bibitem{ls02}
{\sc T.~Lu and S.~Shiou}, {\em Inverses of 2 $\times$ 2 block matrices},
  Computers and Mathematics with Applications, 43 (2002), pp.~119--129.

\bibitem{lms91}
{\sc I.~Lustig, R.~Marsten, and D.~Shannon}, {\em Computational experience with
  a primal-dual interior-point method for linear programming}, Linear Algebra
  and Its Applications, 152 (1991), pp.~192--222.

\bibitem{lms92a}
\leavevmode\vrule height 2pt depth -1.6pt width 23pt, {\em On implementing
  {Mehrotra's} predictor-corrector interior-point method for linear
  programming}, SIAM Journal on Optimization, 2 (1992), pp.~432--449.

\bibitem{Mehrotra92}
{\sc S.~Mehrotra}, {\em On the implementation of a primal-dual interior point
  method}, SIAM Journal on Optimization, 2 (1992), pp.~575--601.

\bibitem{nww09}
{\sc J.~Nocedal, A.~Wachter, and R.~A. Waltz}, {\em Adaptive barrier update
  strategies for nonlinear interior methods}, SIAM Journal on Optimization, 19
  (2009), pp.~1674--1693.

\bibitem{tp98}
{\sc T.~Plantenga}, {\em A trust region method for nonlinear programming based
  on primal interior-point techniques}, SIAM Journal on Optimization, 20
  (1998), p.~282^^e2^^80^^93305.

\bibitem{twbul02}
{\sc A.~L. Tits, A.~Wachter, S.~Bakhtiarl, T.~J. Urban, and C.~T. Lawrence},
  {\em A primal-dual method for nonlinear programming with strong global and
  local convergence properties}, Mathematical Programming, 8 (1998),
  pp.~1132--1152.

\bibitem{uuv04}
{\sc M.~Ulbrich, S.~Ulbrich, and L.~N. Vicente}, {\em A globally convergent
  primal-dual interior-point filter method for nonlinear programming},
  Mathematical Programming, 100 (2004), pp.~379--410.

\bibitem{vs99}
{\sc R.~Vanderbei and D.~Shanno}, {\em An interior-point algorithm for
  nonconvex nonlinear programming}, Computational Optimization and
  Applications, 13 (1999), pp.~231--252.

\bibitem{wright97}
{\sc S.~Wright}, {\em Primal-Dual Interior-Point Methods}, SIAM, Philadelphia,
  1997.

\bibitem{ylz17}
{\sc X.~Yang, H.~Liu, and Y.~Zhang}, {\em An arc-search
  infeasible-interior-point method for symmetric optimization in a wide
  neighborhood of the central path}, Optimization Letters, 11 (2017),
  pp.~135--152.

\bibitem{yang11}
{\sc Y.~Yang}, {\em A polynomial arc-search interior-point algorithm for convex
  quadratic programming}, European Journal of Operational Research, 215 (2011),
  p.~25^^e2^^80^^9338.

\bibitem{yang13}
{\sc Y.~Yang}, {\em A polynomial arc-search interior-point algorithm for linear
  programming}, Journal of Optimization Theory and Applications, 158 (2013),
  pp.~859--873.

\bibitem{yang17}
\leavevmode\vrule height 2pt depth -1.6pt width 23pt, {\em Curvelp-a matlab
  implementation of an infeasible interior-point algorithm for linear
  programming}, Numerical Algorithms, 74 (2017), p.~967^^e2^^80^^93996.

\bibitem{yang18}
\leavevmode\vrule height 2pt depth -1.6pt width 23pt, {\em Two computationally
  efficient polynomial-iteration infeasible interior-point algorithms for
  linear programming}, Numerical Algorithms, 79 (2018), p.~957^^e2^^80^^93992.

\bibitem{yy18}
{\sc Y.~Yang and M.~Yamashita}, {\em An arc-search {$O(nL)$}
  infeasible-interior-point algorithm for linear programming}, Optimization
  Letters, 12 (2018), pp.~781--798.

\bibitem{ye97}
{\sc Y.~Ye}, {\em Interior Point Algorithms: Theory and Analysis}, John Wiley
  \& Son, Inc, New York, 1997.

\bibitem{zyzlh19}
{\sc M.~Zhang, B.~Yuan, Y.~Zhou, X.~Luo, and Z.~Huang}, {\em A primal-dual
  interior-point algorithm with arc-search for semidefinite programming},
  Optimization Letters, 13 (2019), pp.~1157--1175.

\end{thebibliography}

\appendix
\def\thesection{Appendix \Alph{section}}
\section{Derivatives}\label{section:derivatives}
In this section, we give notation related to derivatives.
The Hessian matrix of $f : \R^n \to \R$ is 
\begin{align*}
\nabla^2 f(x) = \left[ \begin{array}{cccc}
\frac{\partial^2 f}{\partial x_1 \partial x_1}  &
\frac{\partial^2 f}{\partial x_1 \partial x_2} &
\cdots &
\frac{\partial^2 f}{\partial x_1 \partial x_n} \\
\vdots  & \vdots  & \ddots  & \vdots \\
\frac{\partial^2 f}{\partial x_n \partial x_1} &
\frac{\partial^2 f}{\partial x_n \partial x_2} &
\cdots &
\frac{\partial^2 f}{\partial x_n \partial x_n} 
\end{array} \right] \in \R^{n \times n}.
\end{align*}
The Jacobian for $h : \R^{n} \to \R^m$ is 
\begin{align*}
\nabla h(x)  = \left[ \begin{array}{cccc}
\frac{\partial  h_1}{\partial x_1 }  &
\frac{\partial  h_2}{\partial x_1 }  &
\cdots &
\frac{\partial  h_m}{\partial x_1 }  \\
\vdots  & \vdots  & \ddots  & \vdots \\
\frac{\partial  h_1}{\partial x_n }  &
\frac{\partial  h_2}{\partial x_n }  &
\cdots &
\frac{\partial  h_m}{\partial x_n } 
\end{array} \right] 
=[\nabla h_1(x),   \cdots, \nabla  h_m(x) ] 
\in \R^{n \times m}.
\end{align*}
The Jacobian for $g : \R^{n} \to \R^p$ is
\begin{align*}
\nabla g(x)  = \left[ \begin{array}{cccc}
\frac{\partial  g_1}{\partial x_1 }  &
\frac{\partial  g_2}{\partial x_1 }  &
\cdots &
\frac{\partial  g_p}{\partial x_1 }  \\
\vdots  & \vdots  & \ddots  & \vdots \\
\frac{\partial  g_1}{\partial x_n }  &
\frac{\partial  g_2}{\partial x_n }  &
\cdots &
\frac{\partial  g_p}{\partial x_n }
\end{array} \right] 
=[\nabla g_1(x),   \cdots, \nabla  g_p(x) ]  
\in \R^{n \times p}.
\end{align*}

For the right-hand-side of (\ref{secondOrder}), we use
\begin{eqnarray*}
 \nabla_x^3 L(x,y,z)\dot{x} \dot{x} 
&=& 
\frac{\partial \left( \frac{\partial^2 L(x,y,z)}{\partial x^2} \dot{x} \right) }
{\partial x} \dot{x} \nonumber 
 =  \sum_{i=1}^n  \dot{x}_i \frac{\partial}{\partial x}
\left[ \begin{array}{c}
\frac{\partial^2 L(x,y,z)}{\partial x_1 \partial x_i} \\
\vdots \\
\frac{\partial^2 L(x,y,z)}{\partial x_n \partial x_i}
\end{array} \right] \dot{x} \\
 \nabla_x^2 h(x)\dot{y} \dot{x} 
&=&
\frac{\partial \left( \frac{\partial  h(x)}{\partial x } \dot{y} \right) }
{\partial x} \dot{x} 
 =  \sum_{i=1}^m \dot{y}_i \frac{\partial}{\partial x}
\left[ \begin{array}{c}
\frac{\partial  h_i(x)}{\partial x_1 } \\
\vdots \\
\frac{\partial  h_i(x)}{\partial x_n} 
\end{array} \right] \dot{x}
=  \sum_{i=1}^m \dot{y}_i
\left( \nabla_x^2 h_i(x) \right)  \dot{x} \\
 \nabla_x^2g(x)\dot{z} \dot{x} 
&=&
\frac{\partial \left( \frac{\partial  g(x)}{\partial x } \dot{z} \right) }
{\partial x} \dot{x} 
 =  \sum_{i=1}^n \dot{z}_i \frac{\partial}{\partial x}
\left[ \begin{array}{c}
\frac{\partial  g_i(x)}{\partial x_1 } \\
\vdots \\
\frac{\partial g_i(x)}{\partial x_n} 
\end{array} \right] \dot{x}
 = \sum_{i=1}^n \dot{z}_i
\left( \nabla_x^2 g_i(x) \right)  \dot{x}  \\
\nabla_x^2 h(x)^{\T} \dot{x}  \dot{x}
&=&
\left( \frac{\partial \left( \left( \frac{\partial  h(x)}
	{\partial x } \right)^{\T} \dot{x} \right) }{\partial x} \right)^{\T} \dot{x}
=  
\left[ \begin{array}{c}
\dot{x}^{\T} \left( \nabla_x^2 h_1(x) \right)  \dot{x} \\
\vdots \\
\dot{x}^{\T} \left( \nabla_x^2 h_m(x) \right)  \dot{x} 
\end{array} \right] \\
\nabla_x^2 g(x)^{\T} \dot{x}  \dot{x}
&=&
\left( \frac{\partial \left( \left( \frac{\partial g(x)}
	{\partial x } \right)^{\T} \dot{x} \right) }{\partial x} \right)^{\T} \dot{x}
=  
\left[ \begin{array}{c}
\dot{x}^{\T} \left( \nabla_x^2 g_1(x) \right)  \dot{x} \\
\vdots \\
\dot{x}^{\T} \left( \nabla_x^2 g_p(x) \right)  \dot{x} 
\end{array} \right].
\end{eqnarray*}

\section{The largest step angle}\label{section:ComputeAlpha}

In this section, we give analytical forms 
to compute the largest $\alpha_{w_i}$ and $\alpha_{s_i}$ for each $i$ in (\ref{alpha}).
For simplicity, here, we drop the index $i$ and the iteration number $k$; 
for example, $w_i^k$ is simply written as $w$.
For (\ref{key1a}), we should have
\begin{align*}
w\AN{\alpha}=w - \dot{w}\sin(\alpha)+\ddot{w}(1-\cos(\alpha)) 
\ge \delta w,
\end{align*}
or equivalently,
\begin{equation}
w - \delta w +\ddot{w}
\ge \dot{w}\sin(\alpha)+\ddot{w}\cos(\alpha).
\label{alphai}
\end{equation}
We split this computation into seven cases
by the signs of $\dot{w}$ and $\ddot{w}$.

\noindent{\it Case 1 ($\dot{w}=0$ and $\ddot{w}\ne 0$)}:

If $\ddot{w} \ge -(1-\delta) w$, 
then $w\AN{\alpha} \ge \delta w$ holds
for $\alpha \in [0, \frac{\pi}{2}]$.
If $\ddot{w} \le -(1-\delta)w < 0$, to meet (\ref{alphai}),
we must have 
$\cos(\alpha) \ge \frac{w-\delta w +\ddot{w}}{\ddot{w}} \ge 0$, or,
$\alpha \le \cos^{-1}\left( \frac{w-\delta w +\ddot{w}}
{\ddot{w}} \right)$. Therefore,
\begin{equation*}
\alpha_{w} = \left\{
\begin{array}{ll}
\frac{\pi}{2} & \quad \mbox{if $w-\delta w +\ddot{w} \ge 0$} \\
\cos^{-1}\left( \frac{w-\delta w +\ddot{w}}
{\ddot{w}} \right) & \quad 
\mbox{if $w-\delta w +\ddot{w} \le 0$}.
\end{array}
\right.
\label{case1a}
\end{equation*}
\noindent{\it Case 2 ($\ddot{w}=0$ and $\dot{w}\ne 0$)}:

If $\dot{w} \le (1 - \delta) w$, then $w\AN{\alpha} \ge \delta w$ holds for any 
$\alpha \in [0, \frac{\pi}{2}]$.
If $\dot{w} \ge (1 - \delta) w > 0$,  to meet (\ref{alphai}),
we must have 
$\sin(\alpha) \le \frac{w-\delta w }{\dot{w}}$, or
$\alpha \le \sin^{-1}\left( \frac{w-\delta w }
{\dot{w}} \right)$. Therefore,
\begin{equation*}
\alpha_{w} = \left\{
\begin{array}{ll}
\frac{\pi}{2} & \quad \mbox{if $\dot{w} \le w-\delta w$} \\
\sin^{-1}\left( \frac{w-\delta w }
{\dot{w}} \right) & \quad \mbox{if $\dot{w} \ge w-\delta w$}.
\end{array}
\right.
\label{case2a}
\end{equation*}
\noindent{\it Case 3 ($\dot{w}>0$ and $\ddot{w}>0$)}:

Let 
$\beta = \sin^{-1} \left( \frac{\ddot{w} }
{\sqrt{\dot{w}^2+\ddot{w}^2}} \right)$.
We can express
$\dot{w}=\sqrt{\dot{w}^2+\ddot{w}^2}\cos(\beta)$ and
$\ddot{w}=\sqrt{\dot{w}^2+\ddot{w}^2}\sin(\beta)$.
Then, (\ref{alphai})
can be rewritten as 
\begin{equation}
w-\delta w + \ddot{w} \ge \sqrt{\dot{w}^2+\ddot{w}^2}
\sin(\alpha + \beta).
\label{alphai1}
\end{equation}
If $\ddot{w} + w- \delta w \ge \sqrt{\dot{w}^2+\ddot{w}^2}$, 
then $w\AN{\alpha} \ge \delta w$ holds for any $\alpha \in [0, \frac{\pi}{2}]$.
If $\ddot{w} + w- \delta w \le \sqrt{\dot{w}^2+\ddot{w}^2}$,  
to meet (\ref{alphai1}), we must have 
$\sin(\alpha + \beta) \le \frac{w-\delta w + \ddot{w}}
{\sqrt{\dot{w}^2+\ddot{w}^2}}$, or
$\alpha + \beta \le \sin^{-1}\left( \frac{w-\delta w + \ddot{w} }
{\sqrt{\dot{w}^2+\ddot{w}^2}} \right)$. Therefore,
\begin{equation*}
\alpha_{w} = \left\{
\begin{array}{ll}
\frac{\pi}{2} & \quad \mbox{if $w-\delta w + \ddot{w} \ge 
	\sqrt{\dot{w}^2+\ddot{w}^2}$} \\
\sin^{-1}\left( \frac{w-\delta w + \ddot{w} }
{\sqrt{\dot{w}^2+\ddot{w}^2}} \right) - \sin^{-1}\left( \frac{\ddot{w} } 
{\sqrt{\dot{w}^2+\ddot{w}^2}} \right) & \quad 
\mbox{if $w-\delta w + \ddot{w} \le 
	\sqrt{\dot{w}^2+\ddot{w}^2}$}.
\end{array}
\right.
\label{case3a}
\end{equation*}
\noindent{\it Case 4 ($\dot{w}>0$ and $\ddot{w}<0$)}:

Let 
$\beta = \sin^{-1} \left( \frac{-\ddot{w} }
{\sqrt{\dot{w}^2+\ddot{w}^2}} \right)$.
We can express $\dot{w}=\sqrt{\dot{w}^2+\ddot{w}^2}\cos(\beta)$ and
$\ddot{w}=-\sqrt{\dot{w}^2+\ddot{w}^2}\sin(\beta)$. Then, (\ref{alphai})
can be rewritten as 
\begin{equation}
w-\delta w + \ddot{w} \ge \sqrt{\dot{w}^2+\ddot{w}^2}
\sin(\alpha - \beta).
\label{alphai2}
\end{equation}
If $\ddot{w} + w- \delta w \ge \sqrt{\dot{w}^2+\ddot{w}^2}$, then $w\AN{\alpha} \ge \delta w$ holds for any $\alpha \in [0, \frac{\pi}{2}]$.
If $\ddot{w} +w- \delta w \le \sqrt{\dot{w}^2+\ddot{w}^2}$,  
to meet (\ref{alphai2}), we must have 
$\sin(\alpha - \beta) \le \frac{w-\delta w + \ddot{w}}
{\sqrt{\dot{w}^2+\ddot{w}^2}}$, or
$\alpha - \beta \le \sin^{-1}\left( \frac{w-\delta w + \ddot{w} }
{\sqrt{\dot{w}^2+\ddot{w}^2}} \right)$. Therefore,
\begin{equation*}
\alpha_{w} = \left\{
\begin{array}{ll}
\frac{\pi}{2} & \quad \mbox{if $w-\delta w + \ddot{w} \ge 
	\sqrt{\dot{w}^2+\ddot{w}^2}$} \\
\sin^{-1}\left( \frac{w-\delta w + \ddot{w} }
{\sqrt{\dot{w}^2+\ddot{w}^2}} \right) + \sin^{-1}\left( \frac{-\ddot{w} }
{\sqrt{\dot{w}^2+\ddot{w}^2}} \right) & \quad 
\mbox{if $w-\delta w + \ddot{w} \le 
	\sqrt{\dot{w}^2+\ddot{w}^2}$}.
\end{array}
\right.
\label{case4a}
\end{equation*}
\noindent{\it Case 5 ($\dot{w}<0$ and $\ddot{w}<0$)}:

Let $
\beta = \sin^{-1} \left( \frac{-\ddot{w} }
{\sqrt{\dot{w}^2+\ddot{w}^2}} \right)$.
We can express
 $\dot{w}=-\sqrt{\dot{w}^2+\ddot{w}^2}\cos(\beta)$ and
$\ddot{w}=-\sqrt{\dot{w}^2+\ddot{w}^2}\sin(\beta)$. Then, (\ref{alphai})
can be rewritten as 
\begin{equation}
w-\delta w + \ddot{w} \ge -\sqrt{\dot{w}^2+\ddot{w}^2}
\sin(\alpha +\beta),
\label{alphai4}
\end{equation}
If $\ddot{w} +( w- \delta w) \ge 0$, then
$w\AN{\alpha} \ge \delta w$ holds for any $\alpha \in [0, \frac{\pi}{2}]$.
If $\ddot{w} +( w- \delta w) \le 0$,  
to meet (\ref{alphai4}), we must have 
$\sin(\alpha + \beta) \ge \frac{-(w-\delta w + \ddot{w})}
{\sqrt{\dot{w}^2+\ddot{w}^2}}$, or
$\alpha + \beta \le \pi - \sin^{-1} \left( \frac{-(w-\delta w + 
	\ddot{w})}{\sqrt{\dot{w}^2+\ddot{w}^2}} \right)$. Therefore,
\begin{equation*}
\alpha_{w} = \left\{
\begin{array}{ll}
\frac{\pi}{2} & \quad \mbox{if $w-\delta w + \ddot{w} \ge 0$} \\
\pi - \sin^{-1} \left( \frac{-(w-\delta w + \ddot{w}) }
{\sqrt{\dot{w}^2+\ddot{w}^2}} \right) - \sin^{-1}\left( \frac{-\ddot{w} }
{\sqrt{\dot{w}^2+\ddot{w}^2}} \right) 
& \quad \mbox{if $w-\delta w + \ddot{w} \le 0$}.
\end{array}
\right.
\label{case5a}
\end{equation*}
\noindent{\it Case 6 ($\dot{w}<0$ and $\ddot{w}>0$)}:

Clearly (\ref{alphai}) always holds for any $\alpha \in [0, \frac{\pi}{2}]$. 
Therefore, we can take \begin{equation}
\alpha_{w} = \frac{\pi}{2}.
\end{equation}
\noindent{\it Case 7 ($\dot{w}=0$ and $\ddot{w}=0$)}:

Clearly (\ref{alphai}) always holds for any $\alpha \in [0, \frac{\pi}{2}]$. 
Therefore, we can take \begin{equation}
\alpha_{w} = \frac{\pi}{2}.
\end{equation}
Similar analysis can be performed for (\ref{key1b}), then similar analytical forms are derived for $\alpha_{s}$.

\section{Details on Numerical Results}\label{sec:detailedResults}

Tables~\ref{tb:solvable_3}, \ref{tb:solvable_4} and \ref{tb:solvable_5}
report the objective value, the numbers of iterations,
and the computation time (in seconds) of 
the proposed arc-search algorithm (Algorithm~\ref{mainAlgo1})
and the line-search algorithm (Algorithm~\ref{mainAlgo_line}) for QCQP and Other type problems.
The symbol ``Unattained'' indicates that the algorithms stopped prematurely, mainly because of the numerical errors. 
We excluded the problems that all the three algorithms 
(Algorithm~\ref{mainAlgo1}, Algorithm~\ref{mainAlgo1} with $\ddot{\ddot{v}}$, and Algorithm~\ref{mainAlgo_line})
stopped with ``Unattained''.
\begin{table}
	\footnotesize    
	\centering
	\caption{Results on QCQP problems\label{tb:solvable_3}}
	\begin{tabular}{|c||r|r|r||r|r|r||r|r|r|} \hline
		& \multicolumn{3}{|c||}{arc-search (Algorithm~\ref{mainAlgo1})} & \multicolumn{3}{|c||}{line-search (Algorithm~\ref{mainAlgo_line})} & \multicolumn{3}{|c|}{arc-search with $\ddot{\ddot{v}}$ in (\ref{ddotddotv})} \\ \hline
		Problem & Obj & Iter & Time & Obj & Iter & Time & Obj & Iter & Time  \\ \hline
BT12 & 6.1881 & 4 & 0.014 & 6.1881 & 20 & 0.009 & 6.1881 & 3 & 0.005 \\ \hline 
TRY-B & 0.0000 & 12 & 0.009 & 1.0000 & 10 & 0.004 & 0.0000 & 18 & 0.011 \\ \hline 
BT1 & -1.0001 & 11 & 0.017 & -0.9937 & 21 & 0.013 & -0.9991 & 8 & 0.011 \\ \hline 
BT2 & 0.0326 & 10 & 0.006 & 0.0326 & 22 & 0.009 & 0.0326 & 11 & 0.006 \\ \hline 
BT4 & 4.6075 & 8 & 0.007 & 4.6075 & 20 & 0.008 & 4.6075 & 5 & 0.003 \\ \hline 
BT5 & 967.6665 & 6 & 0.005 & 961.7151 & 22 & 0.009 & 961.7152 & 5 & 0.003 \\ \hline 
BT8 & 1.0000 & 4 & 0.009 & 1.0000 & 19 & 0.022 & 1.0001 & 7 & 0.012 \\ \hline 
HS108 & -0.8661 & 9 & 0.011 & -0.5000 & 22 & 0.013 & \multicolumn{3}{c|}{Unattained} \\ \hline 
HS113 & 24.3061 & 13 & 0.012 & 24.3058 & 11 & 0.006 & 24.3059 & 9 & 0.005 \\ \hline 
HS12 & -30.0000 & 8 & 0.022 & -30.0001 & 15 & 0.017 & -30.0000 & 12 & 0.020 \\ \hline 
HS22 & 0.9999 & 6 & 0.005 & 0.9999 & 5 & 0.002 & 1.0000 & 5 & 0.003 \\ \hline 
HS30 & 0.9999 & 10 & 0.008 & 0.9999 & 9 & 0.004 & 0.9999 & 10 & 0.006 \\ \hline 
HS31 & 5.9994 & 10 & 0.008 & 5.9993 & 9 & 0.004 & 5.9994 & 11 & 0.007 \\ \hline 
HS43 & -44.0003 & 8 & 0.007 & -44.0002 & 11 & 0.006 & -44.0003 & 9 & 0.006 \\ \hline 
HS63 & 961.7152 & 9 & 0.008 & 961.7151 & 7 & 0.003 & 961.7152 & 10 & 0.006 \\ \hline 
HS65 & 0.9535 & 12 & 0.010 & 0.9535 & 10 & 0.006 & 0.9535 & 15 & 0.010 \\ \hline 
HS83 & -30670.0988 & 20 & 0.018 & -30670.0999 & 21 & 0.010 & -30670.0991 & 21 & 0.013 \\ \hline 
MARATOS & -1.0000 & 3 & 0.002 & -1.0000 & 14 & 0.006 & -1.0000 & 3 & 0.002 \\ \hline 
OPTPRLOC & \multicolumn{3}{c||}{Unattained} & \multicolumn{3}{c||}{Unattained} & -16.4211 & 44 & 0.077 \\ \hline 
ORTHREGB & 0.0000 & 1 & 0.002 & 0.0000 & 26 & 0.016 & 0.0000 & 1 & 0.001 \\ \hline 
ZECEVIC3 & 97.3087 & 9 & 0.006 & 97.3086 & 10 & 0.005 & 97.3087 & 9 & 0.005 \\ \hline 
ZECEVIC4 & 7.5574 & 9 & 0.006 & 7.5575 & 7 & 0.003 & 7.5575 & 8 & 0.004 \\ \hline 
HS11 & -8.4988 & 7 & 0.005 & -8.4985 & 13 & 0.006 & -8.4987 & 8 & 0.004 \\ \hline 
HS14 & 1.3933 & 5 & 0.015 & 1.3934 & 9 & 0.011 & 1.3934 & 6 & 0.011 \\ \hline 
HS18 & 5.0000 & 11 & 0.008 & 5.0000 & 14 & 0.007 & 5.0000 & 14 & 0.008 \\ \hline 
HS27 & \multicolumn{3}{c||}{Unattained} & 0.0400 & 24 & 0.012 & 0.0400 & 27 & 0.023 \\ \hline 
HS42 & 13.8579 & 3 & 0.002 & 13.8579 & 19 & 0.008 & 13.8579 & 3 & 0.002 \\ \hline 
HS57 & 0.0306 & 9 & 0.007 & 0.0285 & 16 & 0.012 & 0.0305 & 15 & 0.009 \\ \hline 
BT13 & -0.0001 & 10 & 0.026 & -0.0001 & 17 & 0.026 & \multicolumn{3}{c|}{Unattained} \\ \hline 
CONGIGMZ & \multicolumn{3}{c||}{Unattained} & \multicolumn{3}{c|}{Unattained} & 27.9991 & 20 & 0.011 \\ \hline 
GIGOMEZ1 & -2.9999 & 40 & 0.036 & -3.0000 & 421 & 0.569 & -3.0001 & 72 & 0.093 \\ \hline 
HAIFAM & -45.0004 & 287 & 14.955 & -45.0004 & 1000 & 9.007 & -45.0003 & 1000 & 14.112 \\ \hline 
HAIFAS & -0.4499 & 5 & 0.015 & -0.4501 & 20 & 0.026 & -0.4499 & 6 & 0.011 \\ \hline 
HS10 & -1.0000 & 7 & 0.005 & -1.0001 & 8 & 0.004 & -1.0000 & 9 & 0.005 \\ \hline 
MAKELA1 & -1.4143 & 34 & 0.032 & -1.4143 & 107 & 0.109 & -1.4143 & 17 & 0.014 \\ \hline 
MAKELA2 & 7.1999 & 7 & 0.005 & 7.2000 & 7 & 0.003 & 7.2000 & 7 & 0.004 \\ \hline 
MAKELA3 & 0.0006 & 12 & 0.019 & 0.0000 & 19 & 0.013 & \multicolumn{3}{c|}{Unattained} \\ \hline 
MIFFLIN1 & -0.9999 & 5 & 0.003 & -1.0001 & 6 & 0.003 & -1.0000 & 5 & 0.003 \\ \hline 
MIFFLIN2 & -1.0000 & 7 & 0.005 & -1.0001 & 10 & 0.005 & -1.0001 & 13 & 0.008 \\ \hline 
MINMAXRB & -0.0001 & 332 & 0.331 & -0.0001 & 11 & 0.006 & 0.0000 & 10 & 0.007 \\ \hline 
POLAK4 & \multicolumn{3}{c||}{Unattained} & \multicolumn{3}{c||}{Unattained} & -0.0001 & 365 & 0.388 \\ \hline 
PRODPL0 & 58.7752 & 33 & 0.225 & 58.7769 & 14 & 0.024 & 58.7759 & 21 & 0.058 \\ \hline 
PRODPL1 & 35.7313 & 28 & 0.188 & 35.7281 & 13 & 0.022 & 35.7298 & 17 & 0.048 \\ \hline 
ROSENMMX & -44.0000 & 10 & 0.007 & -44.0001 & 15 & 0.007 & -43.9999 & 10 & 0.005 \\ \hline 
SMMPSF & 1032924.7420 & 31 & 192.294 & 1032924.7330 & 68 & 36.244 & 1032924.7420 & 30 & 29.892 \\ \hline 
SWOPF & 0.0679 & 26 & 0.333 & 0.0679 & 26 & 0.047 & 0.0679 & 19 & 0.051 \\ \hline 
TRUSPYR1 & 11.2255 & 8 & 0.020 & 11.2254 & 12 & 0.014 & 11.2256 & 8 & 0.014 \\ \hline 
TRUSPYR2 & 11.2203 & 9 & 0.009 & 11.2200 & 24 & 0.017 & 11.2204 & 12 & 0.009 \\ \hline 
COOLHANS & 0.0000 & 5 & 0.006 & 0.0000 & 20 & 0.011 & 0.0000 & 8 & 0.005 \\ \hline 
GOTTFR & 0.0000 & 6 & 0.005 & 0.0000 & 18 & 0.010 & 0.0000 & 5 & 0.003 \\ \hline 
HIMMELBC & 0.0000 & 7 & 0.005 & 0.0000 & 21 & 0.009 & 0.0000 & 4 & 0.002 \\ \hline 
HIMMELBE & 0.0000 & 2 & 0.002 & 0.0000 & 18 & 0.007 & 0.0000 & 2 & 0.001 \\ \hline 
HYPCIR & 0.0000 & 4 & 0.003 & 0.0000 & 18 & 0.008 & 0.0000 & 4 & 0.002 \\ \hline 
HS8 & -1.0000 & 6 & 0.004 & -1.0000 & 21 & 0.010 & -1.0000 & 4 & 0.002 \\ \hline 
	\end{tabular}
\end{table}

\begin{table}
	\footnotesize
	\centering
	\caption{Results on  Others \label{tb:solvable_4}}
	\begin{tabular}{|c||r|r|r||r|r|r||r|r|r|} \hline
    & \multicolumn{3}{|c||}{arc-search (Algorithm~\ref{mainAlgo1})} & \multicolumn{3}{|c||}{line-search (Algorithm~\ref{mainAlgo_line})} & \multicolumn{3}{|c|}{arc-search with $\ddot{\ddot{v}}$ in (\ref{ddotddotv})} \\ \hline
    Problem & Obj & Iter & Time & Obj & Iter & Time & Obj & Iter & Time  \\ \hline
ACOPR30 & 576.8530 & 22 & 1.035 & 576.8530 & 122 & 0.805 & 576.8513 & 37 & 0.438 \\ \hline 
ACOPR30 & \multicolumn{3}{c||}{Unattained} & \multicolumn{3}{c||}{Unattained} & 576.8530 & 956 & 11.482 \\ \hline 
ACOPR57 & 41737.7220 & 271 & 53.130 & 41737.7230 & 107 & 2.513 & 41737.7231 & 30 & 1.091 \\ \hline 
ARGAUSS & 0.0000 & 1 & 0.009 & 0.0000 & 1 & 0.007 & 0.0000 & 1 & 0.008 \\ \hline 
BA-L1 & 0.0000 & 4 & 0.036 & 0.0000 & 23 & 0.040 & 0.0000 & 3 & 0.008 \\ \hline 
BA-L1SP & 0.0000 & 9 & 0.139 & 0.0000 & 24 & 0.063 & 0.0000 & 6 & 0.022 \\ \hline 
BT6 & 0.2770 & 7 & 0.006 & 0.2770 & 18 & 0.008 & 0.2770 & 10 & 0.006 \\ \hline 
BT7 & 306.5000 & 26 & 0.026 & 403.9997 & 30 & 0.015 & 360.3798 & 12 & 0.008 \\ \hline 
BT9 & -1.0000 & 16 & 0.015 & -1.0000 & 28 & 0.013 & -1.0000 & 12 & 0.007 \\ \hline 
BT10 & -1.0000 & 4 & 0.003 & -1.0000 & 18 & 0.007 & -1.0000 & 6 & 0.003 \\ \hline 
BT11 & 0.8249 & 6 & 0.005 & 0.8249 & 21 & 0.009 & 0.8249 & 7 & 0.004 \\ \hline 
CANTILVR & \multicolumn{3}{c||}{Unattained} & \multicolumn{3}{c||}{Unattained} & 1.3399 & 11 & 0.007 \\ \hline 
CB2 & 1.9523 & 6 & 0.004 & 1.9521 & 8 & 0.004 & 1.9522 & 7 & 0.004 \\ \hline 
CB3 & 2.0000 & 8 & 0.006 & 1.9999 & 9 & 0.005 & 2.0000 & 7 & 0.004 \\ \hline 
CHACONN1 & 1.9522 & 9 & 0.006 & 1.9521 & 7 & 0.003 & 1.9522 & 6 & 0.003 \\ \hline 
CHACONN2 & 1.9999 & 7 & 0.005 & 1.9999 & 8 & 0.004 & 2.0000 & 7 & 0.004 \\ \hline 
CLUSTER & \multicolumn{3}{c||}{Unattained} & \multicolumn{3}{c||}{Unattained} & 0.0000 & 5 & 0.007 \\ \hline 
CUBENE & 0.0000 & 4 & 0.003 & 0.0000 & 47 & 0.032 & 0.0000 & 4 & 0.003 \\ \hline 
DIPIGRI & 680.6300 & 16 & 0.013 & 680.6299 & 15 & 0.007 & 680.6300 & 15 & 0.009 \\ \hline 
DIXCHLNG & 2471.8978 & 9 & 0.009 & 2471.8978 & 33 & 0.016 & 2471.8978 & 9 & 0.005 \\ \hline 
FLETCHER & \multicolumn{3}{c||}{Unattained} & \multicolumn{3}{c||}{Unattained} & 19.5232 & 14 & 0.017 \\ \hline 
HALDMADS & \multicolumn{3}{c||}{Unattained} & \multicolumn{3}{c||}{Unattained} & 0.0346 & 48 & 0.059 \\ \hline 
HATFLDF & 0.0000 & 6 & 0.005 & 0.0000 & 15 & 0.007 & \multicolumn{3}{c|}{Unattained} \\ \hline 
HATFLDG & 0.0000 & 18 & 0.052 & 0.0000 & 21 & 0.014 & 0.0000 & 5 & 0.006 \\ \hline 
HEART8 & 0.0000 & 117 & 0.584 & 0.0000 & 469 & 1.838 & 0.0000 & 447 & 2.072 \\ \hline 
HELIXNE & 0.0000 & 5 & 0.004 & 0.0000 & 25 & 0.011 & 0.0000 & 8 & 0.004 \\ \hline 
HIMMELBI & -1735.5689 & 20 & 0.496 & -1735.5698 & 18 & 0.090 & -1735.5689 & 20 & 0.180 \\ \hline 
HIMMELBK & 0.0517 & 17 & 0.042 & 0.0516 & 37 & 0.038 & 0.0516 & 58 & 0.087 \\ \hline 
HIMMELP4 & -8.1980 & 42 & 0.029 & -8.1980 & 20 & 0.010 & \multicolumn{3}{c|}{Unattained} \\ \hline 
HONG & 22.5711 & 11 & 0.010 & 22.5711 & 10 & 0.006 & 22.5711 & 8 & 0.006 \\ \hline 
HS100 & 680.6300 & 16 & 0.013 & 680.6299 & 15 & 0.007 & 680.6300 & 15 & 0.008 \\ \hline 
HS100LNP & \multicolumn{3}{c||}{Unattained} & \multicolumn{3}{c||}{Unattained} & 680.6301 & 7 & 0.014 \\ \hline 
HS100MOD & 678.6796 & 22 & 0.018 & 678.6795 & 14 & 0.007 & 678.6795 & 20 & 0.011 \\ \hline 
HS101 & 1808.9319 & 174 & 0.254 & 1808.9335 & 463 & 0.486 & \multicolumn{3}{c|}{Unattained} \\ \hline 
HS104 & 3.9502 & 10 & 0.010 & 3.9501 & 11 & 0.006 & 3.9502 & 10 & 0.007 \\ \hline 
HS111 & \multicolumn{3}{c||}{Unattained} & -45.8493 & 15 & 0.009 & -47.7612 & 16 & 0.014 \\ \hline 
HS111LNP & -43.1482 & 17 & 0.030 & -45.8490 & 19 & 0.009 & -45.8494 & 10 & 0.006 \\ \hline 
HS112 & -47.7611 & 25 & 0.027 & -47.7611 & 19 & 0.011 & -47.7611 & 28 & 0.022 \\ \hline 
HS114 & \multicolumn{3}{c||}{Unattained} & -1770.6936 & 27 & 0.016 & -1770.6934 & 30 & 0.024 \\ \hline 
HS119 & 244.8788 & 14 & 0.018 & 244.8790 & 12 & 0.010 & 244.8788 & 15 & 0.013 \\ \hline 
HS24 & -1.0000 & 9 & 0.007 & 0.0000 & 11 & 0.007 & -1.0001 & 8 & 0.004 \\ \hline 
HS26 & 0.0000 & 12 & 0.008 & 0.0000 & 21 & 0.009 & 0.0000 & 12 & 0.006 \\ \hline 
HS29 & -22.6275 & 7 & 0.005 & 0.0000 & 19 & 0.015 & -22.6274 & 6 & 0.004 \\ \hline 
HS32 & 0.9997 & 5 & 0.015 & 0.9996 & 6 & 0.009 & 0.9997 & 5 & 0.010 \\ \hline 
HS34 & \multicolumn{3}{c||}{Unattained} & -0.8341 & 45 & 0.021 & -0.8341 & 59 & 0.032 \\ \hline 
HS36 & -3300.2088 & 12 & 0.009 & -0.0001 & 10 & 0.005 & -3300.2088 & 12 & 0.008 \\ \hline 
HS37 & 0.0000 & 15 & 0.010 & -0.0001 & 11 & 0.005 & 0.0000 & 1000 & 0.492 \\ \hline 
HS39 & -1.0000 & 16 & 0.015 & -1.0000 & 28 & 0.013 & -1.0000 & 12 & 0.007 \\ \hline 
HS40 & -0.2500 & 3 & 0.002 & -0.2500 & 15 & 0.006 & -0.2500 & 3 & 0.002 \\ \hline 
HS41 & 1.9259 & 8 & 0.006 & 1.9259 & 6 & 0.003 & 1.9259 & 7 & 0.004 \\ \hline 
HS46 & \multicolumn{3}{c||}{Unattained} & \multicolumn{3}{c||}{Unattained} & 0.0000 & 12 & 0.006 \\ \hline 
HS49 & 0.0000 & 12 & 0.008 & 0.0000 & 26 & 0.011 & 0.0000 & 13 & 0.007 \\ \hline 
HS50 & 0.0000 & 4 & 0.003 & 0.0000 & 31 & 0.013 & 0.0000 & 8 & 0.004 \\ \hline 
	\end{tabular}
\end{table}

\begin{table}
	\footnotesize
	\centering
	\caption{Results on  Others (continued) \label{tb:solvable_5}}
	\begin{tabular}{|c||r|r|r||r|r|r||r|r|r|} \hline
    & \multicolumn{3}{|c||}{arc-search (Algorithm~\ref{mainAlgo1})} & \multicolumn{3}{|c||}{line-search (Algorithm~\ref{mainAlgo_line})} & \multicolumn{3}{|c|}{arc-search with $\ddot{\ddot{v}}$ in (\ref{ddotddotv})} \\ \hline
    Problem & Obj & Iter & Time & Obj & Iter & Time & Obj & Iter & Time  \\ \hline	
HS55 & 6.3333 & 6 & 0.013 & 6.3332 & 6 & 0.008 & 6.3333 & 6 & 0.010 \\ \hline 
HS56 & 0.0000 & 14 & 0.013 & 0.0000 & 20 & 0.009 & 0.0000 & 14 & 0.008 \\ \hline 
HS59 & -7.8027 & 42 & 0.031 & -7.8028 & 18 & 0.010 & -6.7495 & 204 & 0.151 \\ \hline 
HS60 & 0.0326 & 11 & 0.009 & 0.0326 & 11 & 0.006 & 0.0326 & 11 & 0.007 \\ \hline 
HS64 & \multicolumn{3}{c||}{Unattained} & \multicolumn{3}{c||}{Unattained} & 6299.6148 & 19 & 0.012 \\ \hline 
HS66 & \multicolumn{3}{c||}{Unattained} & \multicolumn{3}{c||}{Unattained} & 0.5182 & 60 & 0.032 \\ \hline 
HS68 & 0.0000 & 21 & 0.018 & 0.0000 & 10 & 0.005 & 0.0000 & 21 & 0.015 \\ \hline 
HS69 & 0.0040 & 49 & 0.046 & 0.0040 & 183 & 0.166 & 0.0040 & 52 & 0.042 \\ \hline 
HS7 & 1.7844 & 9 & 0.007 & 1.7844 & 27 & 0.016 & 1.7321 & 16 & 0.014 \\ \hline 
HS71 & \multicolumn{3}{c||}{Unattained} & \multicolumn{3}{c||}{Unattained} & 17.0139 & 16 & 0.018 \\ \hline 
HS73 & 29.8944 & 5 & 0.015 & 29.8943 & 7 & 0.009 & 29.8944 & 5 & 0.010 \\ \hline 
HS74 & 5126.4981 & 18 & 0.019 & 5126.4981 & 18 & 0.009 & 5126.4981 & 17 & 0.014 \\ \hline 
HS75 & 5174.1355 & 23 & 0.022 & 5174.1352 & 15 & 0.007 & 5174.1355 & 23 & 0.017 \\ \hline 
HS77 & 0.2415 & 8 & 0.006 & 0.2415 & 18 & 0.008 & 0.2415 & 8 & 0.004 \\ \hline 
HS78 & -2.9197 & 3 & 0.002 & -2.9197 & 20 & 0.008 & -2.9197 & 3 & 0.002 \\ \hline 
HS79 & 0.0788 & 3 & 0.003 & 0.0788 & 19 & 0.008 & 0.0788 & 4 & 0.002 \\ \hline 
HS86 & \multicolumn{3}{c||}{Unattained} & \multicolumn{3}{c||}{Unattained} & -32.3506 & 17 & 0.010 \\ \hline 
HS99 & -831079891.5000 & 8 & 0.007 & -831079891.5000 & 35 & 0.017 & -831079891.5000 & 8 & 0.005 \\ \hline 
HUBFIT & 0.0169 & 5 & 0.003 & 0.0169 & 5 & 0.003 & 0.0169 & 5 & 0.003 \\ \hline 
HYDCAR20 & 0.0000 & 16 & 0.495 & 0.0000 & 20 & 0.066 & 0.0000 & 8 & 0.042 \\ \hline 
HYDCAR6 & 0.0000 & 6 & 0.019 & 0.0000 & 22 & 0.017 & 0.0000 & 4 & 0.005 \\ \hline 
LAKES & 350525.0229 & 35 & 0.481 & 350524.9285 & 60 & 0.100 & 350525.0229 & 43 & 0.136 \\ \hline 
LEAKNET & 8.0448 & 55 & 3.482 & 8.0020 & 38 & 0.187 & 8.0449 & 41 & 0.362 \\ \hline 
LIN & -0.0176 & 5 & 0.004 & -0.0176 & 13 & 0.006 & -0.0176 & 5 & 0.003 \\ \hline 
LOADBAL & 0.4529 & 34 & 0.075 & 0.4531 & 26 & 0.025 & 0.4529 & 34 & 0.046 \\ \hline 
LOOTSMA & 8.0000 & 10 & 0.009 & 7.7990 & 1000 & 0.442 & 8.0000 & 18 & 0.010 \\ \hline 
LSNNODOC & 123.1027 & 18 & 0.029 & 123.1026 & 11 & 0.013 & 123.1027 & 23 & 0.030 \\ \hline 
MADSEN & 0.6164 & 8 & 0.006 & 0.6163 & 11 & 0.005 & \multicolumn{3}{c|}{Unattained} \\ \hline 
MATRIX2 & 0.0001 & 7 & 0.006 & 0.0001 & 11 & 0.007 & 0.0000 & 28 & 0.024 \\ \hline 
METHANB8 & 0.0000 & 2 & 0.007 & 0.0000 & 17 & 0.013 & 0.0000 & 2 & 0.003 \\ \hline 
METHANL8 & 0.0000 & 3 & 0.010 & 0.0000 & 23 & 0.018 & 0.0000 & 3 & 0.004 \\ \hline 
MINMAXBD & 115.7064 & 33 & 0.041 & 115.7064 & 50 & 0.028 & 115.7064 & 36 & 0.024 \\ \hline 
MWRIGHT & 42.0461 & 8 & 0.007 & 42.0461 & 24 & 0.010 & 42.0461 & 5 & 0.003 \\ \hline 
ODFITS & -2380.0268 & 48 & 0.038 & -2380.0268 & 20 & 0.010 & -2380.0268 & 14 & 0.008 \\ \hline 
POLAK1 & 2.7183 & 9 & 0.006 & 2.7182 & 13 & 0.006 & 2.7182 & 8 & 0.004 \\ \hline 
POLAK2 & 54.5982 & 7 & 0.006 & 54.5981 & 10 & 0.005 & 54.5982 & 7 & 0.004 \\ \hline 
POLAK3 & 5.9329 & 138 & 0.178 & 5.9329 & 12 & 0.006 & 5.9329 & 11 & 0.007 \\ \hline 
POLAK5 & 49.9999 & 46 & 0.038 & 49.9999 & 10 & 0.004 & 50.0000 & 7 & 0.004 \\ \hline 
POLAK6 & \multicolumn{3}{c||}{Unattained} & \multicolumn{3}{c||}{Unattained} & -44.0000 & 22 & 0.013 \\ \hline 
POWELLBS & 0.0000 & 6 & 0.004 & 0.0000 & 17 & 0.008 & 0.0000 & 40 & 0.031 \\ \hline 
RECIPE & 0.0000 & 5 & 0.004 & 0.0000 & 19 & 0.008 & 0.0000 & 9 & 0.005 \\ \hline 
RES & 0.0000 & 29 & 0.042 & 0.0000 & 18 & 0.011 & 0.0000 & 29 & 0.028 \\ \hline 
SINVALNE & 0.0000 & 7 & 0.005 & 0.0000 & 28 & 0.016 & 0.0000 & 3 & 0.002 \\ \hline 
SMBANK & -7129292.0000 & 55 & 2.151 & -7129292.0000 & 196 & 1.709 & -7129292.0000 & 64 & 0.834 \\ \hline 
SYNTHES1 & 0.7573 & 7 & 0.006 & 0.7573 & 7 & 0.004 & 0.7573 & 7 & 0.004 \\ \hline 
SYNTHES2 & -0.5639 & 9 & 0.012 & -0.5636 & 9 & 0.006 & -0.5638 & 9 & 0.008 \\ \hline 
SYNTHES3 & 15.0732 & 13 & 0.020 & 15.0733 & 10 & 0.007 & 15.0730 & 10 & 0.010 \\ \hline 
TRIGGER & 0.0000 & 6 & 0.020 & 0.0000 & 1000 & 4.509 & 0.0000 & 10 & 0.026 \\ \hline 
TRIMLOSS & \multicolumn{3}{c||}{Unattained} & 9.0559 & 142 & 1.778 & 9.0599 & 253 & 5.672 \\ \hline 
TWOBARS & 1.5086 & 6 & 0.004 & 1.5084 & 6 & 0.003 & 1.5086 & 6 & 0.003 \\ \hline 
WATER & 10549.3616 & 31 & 0.131 & 10549.3602 & 34 & 0.064 & 10549.3616 & 28 & 0.073 \\ \hline 
ZY2 & 7.8905 & 7 & 0.007 & 7.8904 & 7 & 0.005 & 7.8904 & 8 & 0.006 \\ \hline 
ACOPR118 & 129660.2236 & 203 & 323.214 & 129660.2294 & 54 & 10.248 & 129660.2319 & 32 & 10.623 \\ \hline 
ALSOTAME & 0.0821 & 8 & 0.005 & 0.0821 & 7 & 0.004 & 0.0821 & 7 & 0.004 \\ \hline 	\end{tabular}
\end{table}

\end{document}